\documentclass[11pt,letterpaper]{article}

\usepackage{algpseudocode}
\usepackage{algorithm}
\usepackage{amsfonts}
\usepackage{amsmath}
\usepackage{amssymb}
\usepackage{amstext}
\usepackage{amsthm}
\usepackage[toc,page]{appendix}
\usepackage{authblk}
\usepackage{booktabs}
\usepackage{cite}
\usepackage{datetime}
\usepackage{enumitem}
\usepackage{etoolbox}
\usepackage{framed}
\usepackage{fullpage}
\usepackage[margin=0.7in]{geometry}
\usepackage{graphics}
\usepackage{graphicx}
\usepackage[latin1]{inputenc}
\usepackage{lineno}
\usepackage{mathrsfs}
\usepackage{multirow}
\usepackage{pdflscape}
\usepackage{pgfplots}
\usepackage{rotating}
\usepackage{setspace}
\usepackage{subcaption}
\usepackage{tikz}
\usepackage{times}
\usepackage{units}
\usepackage{url}
\usepackage{xcolor}
\usepackage{hyperref}
\usepackage[font=small]{caption}

\apptocmd{\sloppy}{\hbadness 10000\relax}{}{}

\usetikzlibrary{shapes,arrows,plotmarks,matrix,positioning,fit,calc,3d}

\newcommand{\iid}{\stackrel{\mathrm{iid}}{\sim}}

\newtheorem{theorem}{Theorem}[section]
\newtheorem{lemma}[theorem]{Lemma}
\newtheorem{proposition}[theorem]{Proposition}
\newtheorem{corollary}[theorem]{Corollary}

\newtheorem{mydef}[theorem]{Definition}

\DeclareMathOperator*{\esssup}{ess\,sup}

\DeclareMathOperator*{\Glim}{\Gamma\text{-}\lim}
\def\eps{\epsilon}

\def\({\left(}
\def\({\right)}

\newcounter{num}
\renewcommand{\thenum}{\arabic{num}}

\newlength\figureheight
\newlength\figurewidth

\begin{document}
\title{Asymptotic Analysis of the Ginzburg-Landau Functional on Point Clouds}
\author[1]{Matthew Thorpe}
\author[2]{Florian Theil}
\affil[1]{Carnegie Mellon University, Pittsburgh, PA 15213, USA}
\affil[2]{University of Warwick, Coventry, CV4 7AL, UK}
\date{November 2016}
\maketitle

\begin{abstract}
The Ginzburg-Landau functional is a phase transition model which is suitable for classification type problems.
We study the asymptotics of a sequence of Ginzburg-Landau functionals with anisotropic interaction potentials on point clouds $\Psi_n$ where $n$ denotes the number data points.
In particular we show the limiting problem, in the sense of $\Gamma$-convergence, is related to the total variation norm restricted to functions taking binary values; which can be understood as a surface energy.
We generalize the result known for isotropic interaction potentials to the anisotropic case and add a result concerning the rate of convergence.
\end{abstract}

\section{Introduction \label{sec:Intro}}

\subsection{Finite Dimensional Modeling \label{subsec:Intro:Finite}}

In the age of `big data' the mathematical modeller is often without a physical model and instead uses a data driven approach for which graphical models are a powerful tool.
Graphical based modelling techniques are used across a very broad spectrum of problems from social science type problems, such as identifying communities~\cite{porter09,danon05,vangennip13,wasserman94,fortunato10}, to image segmentation~\cite{huiyi12,bertozzi12}, to cell biology~\cite{caldarelli07}, to modelling the world wide web~\cite{boccaletti06,caldarelli07,broder00,faloutsos99} and many more.
Anisotropic models, studied in this paper, have found applications in cosmological models~\cite{hennawi07,linder14}, modelling outbreaks of disease~\cite{li12} and image recognition~\cite{worsley99}.

Graphical models are based upon pairwise similarities which practitioners can design based on expert knowledge.
The measure of similarity (on pairs) then defines a geometry (on a data set).
The motivation in this paper is to consider a graphical approach to the classification problems.
Given a measure of similarity we wish to define a labelling using the geometry of the graph.

The problem is given data $\Psi_n=\{\xi_i\}_{i=1}^n\subset X$, where $X\subset \mathbb{R}^d$, find $\mu: \Psi_n\to \mathbb{R}$ that labels each data point.
The labelling is constructed so that $\mu(\xi_i)=0$ means that $\xi_i$ is associated with the class 0 and $\mu(\xi_i)=1$ means that $\xi_i$ is associated with the class 1.
For a finite number of observations we allow a soft classification however the scaling is chosen such that in the data rich limit classifiers are binary valued.
The motivation for our approach is to validate approximating the hard classification problem by a soft classification problem.
The soft classification problem is in general numerically easier~\cite{garcia15} and therefore more appealing to the practitioner.
However one also wants to be precise in regards to which class a data point belongs.
Minimizers of the Ginzburg-Landau functional are used as a classification tool~\cite{vangennip12a} in order to allow for phase transitions which allow a soft classification approach whilst also penalizing states that are not close to a hard classification.

Another important application for this work is in designing classifiers.
By not assuming that the model is isotropic we allow greater flexibility which allows one to choose some features as more important than others.
The next subsection contains a simple example which shows how the design choice can affect the classification.

Assessing the validity of such an approach is of high importance.
This is especially true as there is no natural link between the data generating process and the choice of classifier.
In particular we argue that although using the Ginzburg-Landau functional is a good choice due to its phase transition properties it is by no means the only available option.
When one can make a connection between classifier and data generating process, e.g. maximum likelihood estimators, then this link motivates the methodology.
Without such connection one needs to do more, such as show the estimators have desirable properties, in order to justify the approach.
Other approaches that use classifiers that are detached from the data generating process include~\cite{thorpe15} where the authors prove the convergence of the $k$-means method using similar variational techniques.

An important criterion for validating the model is the behaviour in the large data limit.
When increasing the size of the data set one should expect to see stability in classifiers.
In particular this requires convergence in the large data limit and the existence of a limiting (data  rich) model.
When one has a data generating model, i.e. there is some truth, then one can talk about consistency.
In the situation considered in this paper there is no truth so instead we use solutions to the limiting model.
Knowledge of the limiting model gives an insight into what features one should expect for estimates from the finite data problem.
In particular this paper considers three questions:
\begin{enumerate}
\item[(P1)] Do estimated classifiers $\mu_n$ converge as $n\to \infty$?
\item[(P2)] Can we attach some meaning to any limit of $\mu_n$?
I.e. does there exists a limiting problem?
\item[(P3)] Can we characterize the rate of convergence of estimators?
\end{enumerate}
The primary results of this paper concern the first two questions.
It is shown that estimates $\mu_n$ converge to the solution of a limiting problem.
Furthermore solutions to the limiting problem are binary valued which means we expect estimates $\mu_n$ for large $n$ to be approximately hard classifiers.
For the third question we give some preliminary results into characterizing the rate in a simplified example.
We believe these results will hold under more generality than stated here and it is the objective of ongoing work to extend them.

\begin{figure}
\centering
\includegraphics[width=6cm,height=6cm]{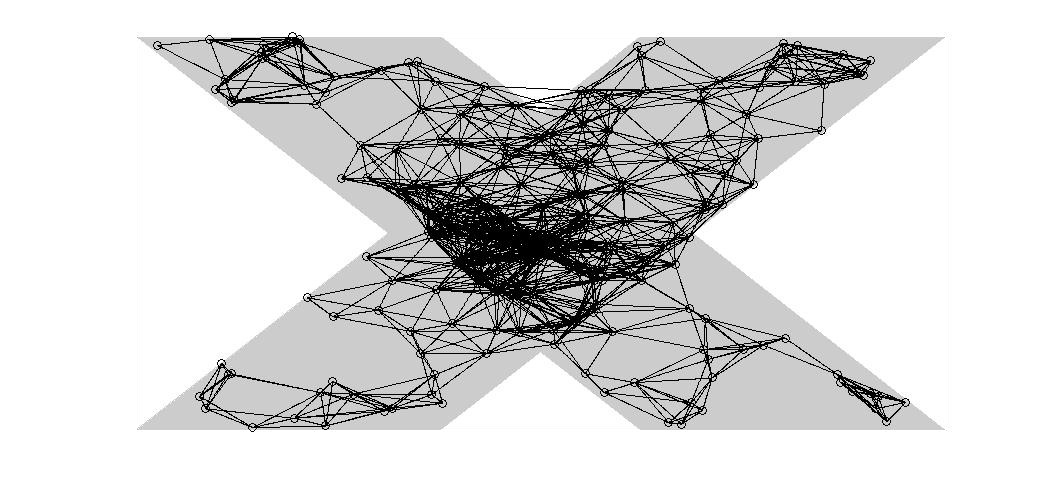}
\caption{An example graph. For the classifier estimates see Figure~\ref{fig:Intro:Finite:Minimizers}.}\label{fig:Intro:Finite:ExampleGraph}
\end{figure}

Our approach is motivated by~\cite{alberti98,trillos15,vangennip12a}.
Classifiers are constructed as the solution of a variational problem which is common in statistical problems, e.g. maximum likelihood and maximum-a-posterior problems.
In particular minimizers of the Ginzburg-Landau functional, a phase transition model popular in material science and image segmentation, are used as classifiers.
In the context of the 2-class classification problem the two phases are the classes and the phase transition corresponds to the set $\left\{ x \in \Psi_n \, : \, \mu_n(x) \in (c,1-c) \right\}$ for some $c\in (0,1)$, e.g. $c=0.1$.
This is the subset of $\Psi_n$ that are not strongly associated with either class.

Classifiers $\mu_n:\Psi_n\to \mathbb{R}$ are constructed as follows.
Let $V:\mathbb{R}\to [0,\infty)$ be a potential such that states taking the value 0 or 1 is favoured.
For example $V(t) = t^2(t-1)^2$.
A graph is constructed by taking the vertices as the set $\Psi_n=\{\xi_i\}_{i=1}^n \subset \mathbb{R}^d$ and weighting edges
\begin{equation} \label{eq:Intro:Finite:W}
W_{ij} = \eta_\epsilon(\xi_i-\xi_j) := \frac{1}{\epsilon^d} \varphi\circ \pi\left(\frac{\xi_i-\xi_j}{\epsilon}\right)
\end{equation}
where $\pi: X\to \mathbb{R}$ is a given one-dimensional map so that $\pi(\xi_i)$ represents a feature of $\xi_i$ and $\varphi:\mathbb{R}\to [0,\infty)$ penalizes the difference $\pi(\xi_i-\xi_j)$.
We say that there is an edge between $\xi_i$ and $\xi_j$ if $W_{ij}>0$, for example see Figure~\ref{fig:Intro:Finite:ExampleGraph}.
For a function $\mu_n$ on $\Psi_n$ the graph energy $\mathcal{E}_n(\mu_n) \in [0,\infty]$ is defined by
\begin{equation} \label{eq:Intro:Finite:En}
\mathcal{E}_n(\mu_n) = \frac{1}{\epsilon_n} \frac{1}{n} \sum_{i=1}^n V(\mu_n(\xi_i)) + \frac{1}{\epsilon_n} \frac{1}{n^2} \sum_{i,j} W_{ij} |\mu_n(\xi_i)-\mu_n(\xi_j)|.
\end{equation}
Our classifier is given as the minimizer of~$\mathcal{E}_n$.

We call the map $\pi:X\to \mathbb{R}$ the feature projection as it allows the practitioner to include feature selection of the data $\xi_i$. 
For example one may decide that two data points should be considered similar based on the pairwise difference.
In this case an appropriate choice would be the weighted Euclidean distance $\pi(x) = \sqrt{\sum_{i=1}^d w_i|x_i|^2}$.
The isotropic case would correspond to weights $w_i=1$.
Other choices could be to include correlations between dimensions, for example $\pi(x) = \sqrt{\sum_{i,j=1}^d w_{ij}|x_i||x_j|}$.

The authors of~\cite{trillos15} study the asymptotic properties of the graph total variation defined by
\begin{equation} \label{eq:Intro:Finite:TVG}
GTV_n(\mu_n) := \frac{1}{\epsilon_n} \frac{1}{n^2} \sum_{i,j} W_{ij} |\mu_n(\xi_i)-\mu_n(\xi_j)|
\end{equation}
when $W_{ij}$ is isotropic, i.e. $\pi(x) = |x|$.
In the special case that $\mu_n(\xi_i)\in\{0,1\}$ this reduces to the graph cut of $\Psi_n$, i.e. if $(\mu_n)^{-1}(0)=A_0$ and $(\mu_n)^{-1}(1) = A_1$ then
\[ GTV_n(\mu_n) = \frac{1}{\epsilon_n} \frac{1}{n^2} \sum_{\xi_i\in A_0} \sum_{\xi_j\in A_1} W_{ij}. \]
In particular the authors in~\cite{trillos15} show the $\Gamma$-convergence of $GTV_n$ to a weighted total variation $TV(\cdot;\rho,\eta)$ given by
\[ TV(\mu;\rho,\eta) := \sigma_\eta \sup \left\{ \int_X \mu \mathrm{div}(\phi) \, \mathrm{d} x \, : \, |\phi(x)| \leq \rho^2(x) \, \forall x \in X, \phi\in C_c^\infty(X;\mathbb{R}^d) \right\}, \]
and $L^1$-compactness for any sequence $\mu_n$ with $\sup_n (GTV_n(\mu_n) + \|\mu_n\|_{L^1}) < \infty$.

We wish to allow for soft classification and the total variation term alone is not enough to be able to do this informatively.
The classification approach is made more robust by including a first order term $V:\mathbb{R}\to [0,\infty)$ which penalizes associating a data point to more than one class.
See, for example, Figure~\ref{fig:Intro:Finite:Minimizers} for a comparison.
It is not trivial that the convergence results in~\cite{trillos15} will survive adding a penalty term.

\begin{figure}
\centering
\begin{subfigure}[b]{0.3\textwidth}
\includegraphics[width=4cm,height=4cm]{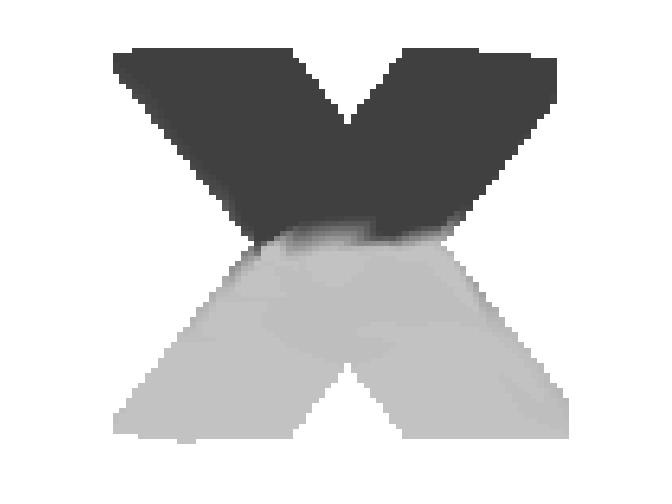}
\end{subfigure}
~
\begin{subfigure}[b]{0.3\textwidth}
\includegraphics[width=4cm,height=4cm]{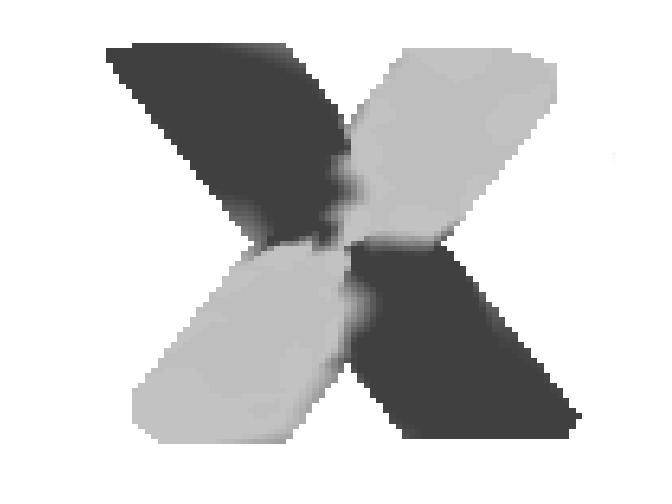}
\end{subfigure}

\begin{subfigure}[b]{0.3\textwidth}
\includegraphics[width=4cm,height=4cm]{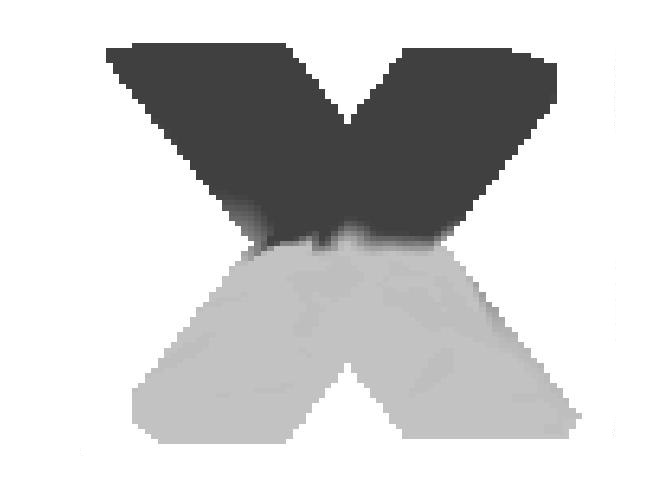}
\end{subfigure}
~
\begin{subfigure}[b]{0.3\textwidth}
\includegraphics[width=4cm,height=4cm]{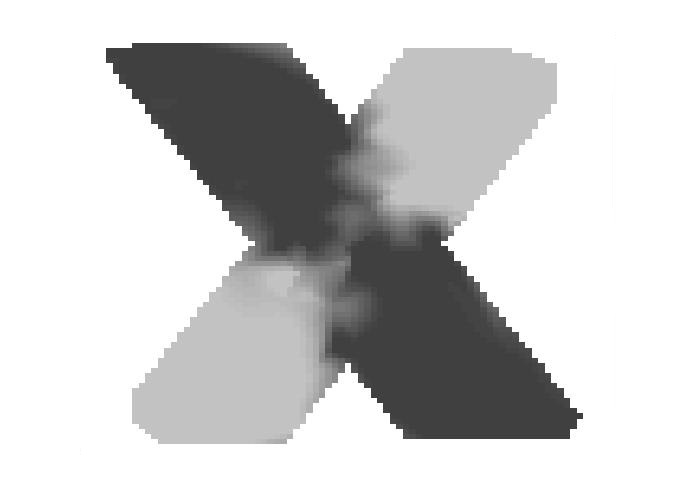}
\end{subfigure}
\caption{
The top row shows the minimizers of $\mathcal{E}_n$ and the bottom row shows the minimizers of $GTV_n$ for the graph given in Figure~\ref{fig:Intro:Finite:ExampleGraph} conditioned on the node closest to each corner taking either 0 or 1.
The left column is conditioned to have 0 in the bottom corners and 1 in the top corners.
The right column has 0 in the bottom left and top right corners and 1 in the top left and bottom right corners. There is very little difference between the outputs on the left but on the right the $GTV_n$ term fails to pick out the singularity at the center.
}
\label{fig:Intro:Finite:Minimizers}
\end{figure}

Finding minimizers of $\mathcal{E}_n$ is also an important problem but is not addressed in this paper.
We instead refer to \cite{bresson12,bresson13} for numerical methods.

\subsection{Example: Classification Dependence on the Choice of \texorpdfstring{$\eta$}{eta} \label{subsec:Intro:Ex}}

Through a toy problem we demonstrate how the interaction potential can be used to pick out features of the practitioners choice.
Data are points $\xi_i=(\xi_{1i},\xi_{2i})\in\mathbb{R}^2$ generated from four classes.
For a fixed $\alpha$ the feature projection $\pi:\mathbb{R}^2 \to [0,\infty)$ is defined by the weighted Euclidean norm
\[ \pi(\xi) = \sqrt{(1-\alpha) \xi_1^2 + \alpha \xi_2^2}. \]
For $\alpha\approx 1$ the classifiers are dominated by differences in the first coordinate whilst for $\alpha\approx 0$ the classifiers are dominated by differences in the second coordinate.
More precisely, let
\begin{align*}
\mu_1(\xi_i) & = \left\{ \begin{array}{ll} 1 & \text{if } \xi_{2i} \leq c_1 \\ 0 & \text{otherwise} \end{array} \right. \\
\mu_2(\xi_i) & = \left\{ \begin{array}{ll} 1 & \text{if } \xi_{1i} \leq c_2 \\ 0 & \text{otherwise.} \end{array} \right.
\end{align*}
Then define
\[ \Delta\mathcal{E}_n = \mathcal{E}_n(\mu_1) - \mathcal{E}_n(\mu_2). \]
The results are given in Figure~\ref{fig:Intro:Ex:Data}.

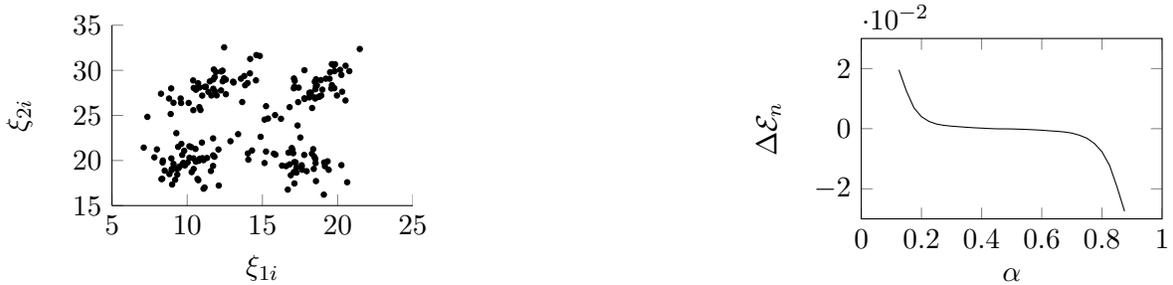
\begin{figure}[ht]
\centering
\begin{subfigure}[b]{0.45\textwidth}
\setlength\figureheight{2.4cm}
\setlength\figurewidth{4cm}
\centering
%
%
%
%
\begin{tikzpicture}

\begin{axis}[%
width=\figurewidth,
height=\figureheight,
scale only axis,
xmin=5,
xmax=25,
ymin=15,
ymax=35,
xlabel={$\xi_{1i}$},
ylabel={$\xi_{2i}$},
axis x line*=bottom,
axis y line*=left
]
\addplot [
color=black,
mark size=1.0pt,
only marks,
mark=*,
mark options={solid},
forget plot
]
table[row sep=crcr]{
10.1454775351944 21.5784029858823\\
9.85198513615763 19.4584366862316\\
10.7634520251769 19.9640338486672\\
7.81763632208053 20.3506499044257\\
8.50527524449953 18.8749666590756\\
11.7195440966756 19.3971157804164\\
14.0637714188745 20.0925821183902\\
8.99832629991091 17.3406294783807\\
9.20109404275693 17.8625567954474\\
11.085114247581 16.8714828097818\\
9.802819619058 21.0797717495514\\
8.28709570661618 17.9427109751534\\
9.28617537853287 19.1167812519677\\
11.0855751642994 20.0410169835516\\
10.9771901189658 21.9909366251986\\
12.0977034219099 17.2233816663058\\
9.67091685286115 20.6202912927611\\
7.98721575286231 21.2326098752796\\
9.66452098393983 19.5230921352996\\
11.7222836406613 22.4535032979681\\
10.2674261154024 21.4821332030238\\
10.4294688513925 19.061758115426\\
8.94877976282042 20.2337068122422\\
11.0139240261232 20.2585277622563\\
8.97010865966609 18.932433771544\\
8.36394034039859 19.7823486604429\\
10.7400447905052 17.8170193440694\\
8.98261102695603 19.0735518863846\\
7.11098347048211 21.4330347455332\\
11.7232028632219 20.5735063788355\\
10.0315573746835 19.7786152259513\\
9.47499553736306 19.3131829294177\\
9.27935134288355 23.0318503233304\\
10.147528843465 20.3037720459748\\
8.34142086229608 17.9650971862204\\
11.5751079653239 18.8360249031262\\
9.09063595869931 19.6193947165925\\
12.8748992437319 22.1379445710451\\
11.1563054829493 17.0073055066412\\
10.4199823276651 20.240294526587\\
8.38238504911821 19.9971028890555\\
12.0760027834058 21.2200130751659\\
11.8113179540804 20.4173078390204\\
10.5714787683188 19.8739444501002\\
9.39521216456401 21.5141685053032\\
9.73915589016471 19.7376660516104\\
8.81676978741258 18.496677145791\\
10.5510915452835 21.2843179362588\\
9.34201332567616 18.4424586052826\\
9.61419937532222 21.8412376316449\\
14.3093366683885 21.1114845435313\\
10.6835910528828 17.953729973938\\
13.3839769427905 22.9215890334553\\
11.3187369489854 20.2821387731203\\
14.0317867095482 20.7838967251802\\
9.42572416403748 19.1510461268531\\
10.8417983162711 20.200862569104\\
10.2980952748618 18.7097251388024\\
};
\addplot [
color=black,
mark size=1.0pt,
only marks,
mark=*,
mark options={solid},
forget plot
]
table[row sep=crcr]{
17.7806131651598 20.6171857413157\\
20.2505140749207 19.487292581389\\
16.8848305354838 18.3661909441054\\
18.5598673173893 17.7184505096065\\
16.8049649488618 19.572077155056\\
17.2643978369117 18.6661681529019\\
18.2887464685663 20.7584181520254\\
17.2681309457703 19.2048716570789\\
16.949816131694 21.3818488336435\\
19.1058258255468 19.7869464494887\\
17.2049322801205 19.7993710862117\\
16.5733108355867 19.3758575739082\\
17.0931212840095 18.7231823583871\\
17.0000136702987 21.1785033071587\\
16.3036257484632 19.4346892567944\\
20.6360225771026 17.5888740904175\\
19.0825616349992 16.2159785262911\\
17.1160186270963 19.8158797095813\\
18.4495555200491 20.4124193487677\\
15.8646357672253 20.6463611258962\\
18.5153732845244 20.1949966510702\\
19.1531427677943 19.3717876010168\\
16.6877785811333 16.7720337004001\\
16.7339851200373 20.8718875825303\\
14.8810836476469 22.6273749766853\\
18.4138449924647 19.0689682354616\\
17.6468459319147 18.9700666909013\\
17.9422794922584 18.8127786812278\\
17.1279359791893 17.4572938042914\\
15.7615907021083 20.7604140051346\\
15.3719498759988 24.6557784946368\\
17.6123195505466 19.3451871787195\\
15.1406228106664 19.7210886508643\\
17.0464619515085 20.8103740785287\\
18.3763776340932 21.258739396522\\
19.3338394766166 19.9104528694136\\
15.239202925877 20.9922204478614\\
18.5326948975688 19.7087063706826\\
19.4072284848048 18.9657319026549\\
17.6077548206476 19.4773997101316\\
};
\addplot [
color=black,
mark size=1.0pt,
only marks,
mark=*,
mark options={solid},
forget plot
]
table[row sep=crcr]{
11.1605910992255 28.1645252091497\\
14.1703532648574 29.6655435842378\\
12.5688091799319 28.9674213655699\\
10.9986182378446 27.1975217303461\\
11.6081353329437 27.2101703260863\\
12.450251527146 29.1168286438823\\
10.7269706122044 28.5630778293036\\
14.1874403922719 31.2641161642562\\
14.8261429077116 31.5956828771226\\
13.8533337157247 28.3460043246556\\
14.5705770344632 28.8884171416211\\
9.08973691808539 26.4043959984731\\
10.5796539804762 27.8465603153562\\
13.573252380519 29.0732424346368\\
13.0506542549449 28.7410734733499\\
9.59491743021193 26.3742964043527\\
11.8571903567053 27.6346389386906\\
10.8833394481904 25.6037478058056\\
10.3958335036265 28.8793083676942\\
12.2493040448237 27.7746229203608\\
11.733384363975 27.9599901324445\\
12.5733865322654 27.349069181293\\
12.4570419278169 32.5444162431666\\
14.5964570578476 31.6940594533981\\
10.4327358287395 28.0484498062901\\
12.3573849075125 28.7831123526512\\
10.4087105986419 25.5809722972505\\
11.91032624415 29.8154650113357\\
11.4993028102215 28.7717777040771\\
11.7701628533659 30.0936631984305\\
9.5649855021066 26.8751415583119\\
11.2400576495297 28.1843686064285\\
8.90309089184985 25.155908244\\
13.7951704452538 29.3727861824897\\
10.5800570479441 27.9705987657784\\
10.8180562467814 28.0965079611493\\
13.6565235520518 26.4822871675204\\
8.25884459870477 27.4060752619959\\
13.0828888465913 28.6415133612451\\
12.263563139052 29.8767703923367\\
11.8260037437765 29.2377044268489\\
8.79304861703094 26.8749864425802\\
14.0114237674355 28.5870388830581\\
10.80689715426 25.9181581904307\\
10.0963591123123 26.3857532716755\\
11.7336101618052 28.9919499693471\\
7.3549459886941 24.8393027830815\\
12.3490555723476 29.9680429763778\\
8.93593973975636 27.9921176410208\\
11.9560383017487 27.2246975431968\\
11.4091020919906 27.6122260826375\\
};
\addplot [
color=black,
mark size=1.0pt,
only marks,
mark=*,
mark options={solid},
forget plot
]
table[row sep=crcr]{
17.1257817171839 28.1867039993688\\
18.7655926945835 27.0834646255691\\
20.2490263709527 29.4892156756865\\
17.8013600616297 26.8250191488877\\
18.9839681901111 27.858861262298\\
19.9536571487301 27.2075398553868\\
18.5124818660683 26.8730148236953\\
17.0871359537003 28.007347151919\\
17.1479302198705 28.8270607590784\\
17.3736992122506 26.4720755011583\\
19.799848570299 29.9125132749391\\
16.2357834814862 24.6248506840159\\
16.8093313519223 25.915953753735\\
20.3019459163098 27.6280416938398\\
18.4828233161308 28.7330886942083\\
17.5203404474172 22.5490437057833\\
20.1660930163897 30.0534101142384\\
19.4963834353542 29.7892462058974\\
15.2418297308721 26.0236119450022\\
19.6182609201962 28.3597549222108\\
17.3367152286835 23.8846295429462\\
18.3060130009379 25.8199829167711\\
18.6212572741508 28.3383885491028\\
18.119447145721 26.9952007051027\\
20.7766540709307 29.9131924670913\\
19.5786527840887 27.9829169079694\\
18.1223407946924 27.5176222406577\\
17.408195279139 28.0783877939683\\
20.5225488003892 26.6579164186885\\
17.0943681480896 29.0971154615656\\
19.8287382661435 30.674403193106\\
18.4583819391316 26.9954220007696\\
18.1462201504098 27.3213183560846\\
17.7891710982061 26.8938963939522\\
15.8407795161202 25.0380612163376\\
18.5277361827434 27.9081964373673\\
19.4640523897363 29.0727010180685\\
15.1429806925411 24.5341763252867\\
21.4709565540875 32.3665639623753\\
19.2178210505242 29.0735520116656\\
19.6129669575273 30.6902304412629\\
19.7789474681808 28.0171382353566\\
18.759298806136 27.0541012586718\\
20.5297089692343 30.5073527284578\\
18.5639422121781 29.0077621438769\\
18.9108473269518 27.2027017581201\\
18.8835842994777 28.9126905021993\\
18.5558000962727 28.2676928380379\\
19.7448655839102 30.5572861321937\\
18.223898342154 27.5334965903209\\
17.7926929340664 30.0123416120173\\
};
\end{axis}
\end{tikzpicture}%
\end{subfigure}
\hfill
\begin{subfigure}[b]{0.45\textwidth}
\setlength\figureheight{2.4cm}
\setlength\figurewidth{4cm}
\centering
%
%
%
%
\begin{tikzpicture}

\begin{axis}[%
width=\figurewidth,
height=\figureheight,
scale only axis,
xmin=0,
xmax=1,
ymin=-0.03,
ymax=0.03,
xlabel={$\alpha$},
ylabel={$\Delta\mathcal{E}_n$}
]
\addplot [
color=black,
solid,
forget plot
]
table[row sep=crcr]{
0.125 0.019575\\
0.15 0.0125375\\
0.175 0.0069375\\
0.2 0.003975\\
0.225 0.002425\\
0.25 0.0015375\\
0.275 0.0011\\
0.3 0.0008375\\
0.325 0.00065\\
0.35 0.0004625\\
0.375 0.0003\\
0.4 0.0001875\\
0.425 5e-005\\
0.45 -5e-005\\
0.475 -6.25e-005\\
0.5 -8.75000000000001e-005\\
0.525 -0.0001625\\
0.55 -0.0002625\\
0.575 -0.000375\\
0.6 -0.0005125\\
0.625 -0.0007125\\
0.65 -0.0008625\\
0.675 -0.0011\\
0.7 -0.001475\\
0.725 -0.002175\\
0.75 -0.0031875\\
0.775 -0.00495\\
0.8 -0.0076875\\
0.825 -0.0121625\\
0.85 -0.0193875\\
0.875 -0.0274625\\
};
\end{axis}
\end{tikzpicture}%
\end{subfigure}
\caption{
The figure on the left shows a realisation of the data.
The interactions are parameterized by a potential $\alpha$ which favours horizontal partitions for $\alpha\approx 0$ and vertical partitions for $\alpha\approx 1$ as shown by the figure on the right.
}
\label{fig:Intro:Ex:Data}
\end{figure}

\subsection{The Limiting Model \label{subsec:Intro:Infinite}}

Rather surprisingly the problem of soft classifications for finite data sets and hard classification in the limit has received relatively little attention in the literature.
However it is well known that for finite data one can recover the $k$-means algorithm (hard classification) from the expectation-maximization algorithm (soft classification) in the zero-variance limit for the Gaussian mixture model and the Dirichlet process mixture model~\cite{kulis12,MacKay02}.

The results of this paper concern the asymptotics of the minimum and minimizers of $\mathcal{E}_n$, where $\epsilon_n\to 0$ as $n\to \infty$.
The advantages of scaling $\epsilon_n$ to zero are two-fold.
The first is that the matrix $W=(W_{ij})_{ij}$ is sparse and therefore we expect the minimization to be numerically less expensive than solving the minimization with a non-sparse matrix (since the sparse minimization has $O(n)$ terms and the non-sparse minimization $O(n^2)$ terms).
The second is to improve resolution of the boundary.
One can think of soft classification as estimating the probability that a data point belongs to a certain class and the hard classification problem as estimating the boundaries where one class is more likely than all others.
By scaling $\epsilon_n\to 0$ it will be shown that the limiting minimization problem is a hard classification.
For example, Figure~\ref{fig:Intro:Infinite:ExampleRes} shows (for a fixed number of data points) improved resolution in the boundary between classes as $\epsilon\to 0$.
See also~\cite{garcia15}.

Assume $X\subset\mathbb{R}^d$ and define $\mathcal{E}_\infty:L^1(X) \to [0,\infty]$ by
\begin{equation} \label{eq:Intro:Infinite:EinftySurface}
\mathcal{E}_\infty(\mu) = \left\{ \begin{array}{ll} \int_{\partial \{\mu = 1\}} \sigma(n(x)) \rho^2(x) \; \mathrm{d} \mathcal{H}^{d-1}(x) & \text{if } \mu\in L^1(X;\{0,1\}) \\ \infty & \text{otherwise} \end{array} \right.
\end{equation}
where $n(x)$ is the outward unit normal for the set $\partial\{\mu=1 \}$, $\mathcal{H}^{d-1}$ is the $d-1$ dimensional Hausdorff measure and
\[ \sigma(\nu) = \int_{\mathbb{R}^d} \varphi(\pi(x)) |x\cdot \nu| \; \mathrm{d} x. \]
It will be shown, in the sense of $\Gamma$-convergence, that $\mathcal{E}_\infty$ is the limiting problem and any sequence such that $\mathcal{E}_n(\mu_n)$ is bounded is precompact in an appropriate topology.
In particular this allows one to apply the results of this paper to infer the consistency of the constrained minimization problem (see Section~\ref{subsec:MainRes:Comments}).

\begin{figure}[ht]
\centering
\begin{subfigure}[b]{0.45\textwidth}
\setlength\figureheight{3cm}
\setlength\figurewidth{5cm}
\centering
%
%
%
%
\begin{tikzpicture}

\begin{axis}[%
width=\figurewidth,
height=\figureheight,
ticks=none,
scale only axis,
xmin=-5,
xmax=10,
ymin=-3,
ymax=4.5
]
\addplot [
color=black,
mark size=1.2pt,
only marks,
mark=*,
mark options={solid,fill=black},
forget plot
]
table[row sep=crcr]{
-4.61354441224525 0.832361464977937\\
-4.54881155210922 2.33882718654275\\
-4.25967345950225 0.0472303796587409\\
-4.14762514910441 -0.606895785658696\\
-4.08577066491708 -0.572165580214948\\
-4.06808555702868 -0.627206745334295\\
-4.03507907733803 1.10874732208355\\
-4.02253474133478 0.511976715695957\\
-3.9342951843913 0.0540746740583324\\
-3.90117871279631 -0.198846543738207\\
-3.82523307079921 0.613628133702281\\
-3.69051836443586 0.285213162034563\\
-3.68924436738692 -0.30263416245437\\
-3.66969410761964 0.658313598849679\\
-3.66117712704449 0.50847272370766\\
-3.51176945994596 0.884055764125925\\
-3.49836008009677 0.573189789389545\\
-3.49756408290977 0.24965429488233\\
-3.49389821329317 -0.453011823665195\\
-3.46645778778749 -0.43191487028127\\
-3.44954259960669 0.274294775641414\\
-3.38822063543571 -0.428028710934555\\
-3.28169317130559 0.606384741137854\\
-3.26688872999568 -0.301550997877725\\
-3.24634463268615 -0.989576112558169\\
-3.24547085566336 -0.665737286673611\\
-3.16527852996571 -0.430819866430574\\
-3.1543548164043 -0.113540629871013\\
-3.11804108897505 -1.90445541188568\\
-2.88367283508239 -2.31342226966951\\
-2.81461036898861 -0.0820649512189811\\
-2.76489631667491 -2.36529251516757\\
-2.74844806272599 -1.69163694142652\\
-2.73017639168586 -0.120274492152151\\
-2.72511768890796 -0.872462560090612\\
-2.65823637881099 -1.35161740327127\\
-2.60162911799044 -1.48491033439179\\
-2.51486429723368 -1.88432095929566\\
-2.44967043071795 -0.686518692680404\\
-2.28980223639312 -1.84942503650549\\
-1.96109579090123 -1.11766905155629\\
-1.55984742445426 -1.80570196376781\\
-1.38058367023737 -0.81496684519942\\
-1.35110851510594 -1.72294323809533\\
-1.35040130194246 -1.88045294243807\\
-1.32476774225315 -0.94185512694711\\
-1.13702734064839 -1.99898029011068\\
-0.87569691456725 -1.64643189167766\\
-0.792546802171743 0.927470716364932\\
-0.782094888031326 -2.38613280533269\\
-0.757045633897663 -1.80944351463797\\
-0.742456846377258 1.00064601529753\\
-0.613573299782455 -1.79779284226345\\
-0.59736714216632 -1.27156278870325\\
-0.595053446060927 0.205626055795326\\
-0.26702064033175 -2.01640471908406\\
-0.258269276567031 -0.227321988455058\\
-0.231462476864986 -1.40576339109543\\
-0.225007912239916 0.917381037592895\\
-0.2035262114635 0.472563188958718\\
-0.142836243607826 0.699012244270581\\
-0.118578543837518 -2.10046356679899\\
-0.0850627575387945 0.535823687738308\\
-0.0569059727056669 1.29229284102957\\
-0.0443805835823826 1.01356195274905\\
0.0217493401691951 -2.42131223805036\\
0.0556254878497774 -2.67579091357577\\
0.176610229060615 -1.81939992389389\\
0.246426045365254 0.0453455468413395\\
0.400045803621939 -2.34387699958838\\
0.407841782737313 -0.735970192237221\\
0.409155255286757 -1.73555908452236\\
0.558051976873671 0.482009371219448\\
0.570170142612265 -1.95687142280684\\
0.576601723681316 1.82260825280297\\
0.601367667417972 1.05067371619356\\
0.703449672681064 -1.71329554997515\\
0.70829433216618 2.10637467493826\\
0.823815329150159 0.150926972994904\\
0.940833017176902 2.02186407760649\\
0.963345889998623 1.86468692814079\\
0.970933992683241 -1.61064556920354\\
0.97977162590643 1.08351323187414\\
0.999773100667601 0.281582345406385\\
1.01455707268897 -2.08625617196336\\
1.09857231360013 1.74619610246882\\
1.1474789991681 -1.98294166220577\\
1.1683686578735 -2.40328734112304\\
1.20259879491837 0.991333437576832\\
1.20260742815256 1.01455147754742\\
1.23158588820431 -2.06952648375271\\
1.28936203003942 0.647415032315287\\
1.29956123250541 -1.83093608204321\\
1.32258212082176 1.639244137448\\
1.38340455002 -1.33640371085686\\
1.39985712333535 2.59684196520614\\
1.42035041816762 -1.68190192170078\\
1.44482528258127 -1.71375569317539\\
1.45434798552229 2.80325907303559\\
1.52415651738386 2.26955173451384\\
1.53933824118 3.24329529372174\\
1.56634979666714 -2.78698074897775\\
1.74984154025752 2.67959489804003\\
1.75180730820424 2.16217199136233\\
1.75771628513257 -1.42499917050143\\
1.81076707621864 2.22241262325692\\
1.81228716644463 2.63009388363283\\
1.86186433660041 1.326027550149\\
1.94776174921112 -0.911198419798203\\
1.95886691295846 2.17805845882628\\
1.96065405347892 1.68623908859185\\
1.97206433544139 -0.419737293592665\\
2.00264474528859 1.95756479254959\\
2.00811445757024 -0.636671205249424\\
2.12957593159012 2.05185853068469\\
2.16927105035048 2.20880088648508\\
2.22125854066778 2.46151997545208\\
2.22898457998684 3.14408526839907\\
2.36720378416185 -2.16218729511017\\
2.50419056777272 -1.85006833245121\\
2.5506901347904 -1.79656560017092\\
2.66104875789974 -0.923211461363686\\
2.77381224971741 -1.54055290370376\\
2.78527901429027 -2.35915787319209\\
2.79845809699079 -0.99347915854763\\
2.82710411064881 -1.42446045982968\\
2.96796484955628 -0.514714395626115\\
2.9765965705947 -0.832412983299476\\
2.98386567561932 -1.1319082875363\\
3.18707066381269 -1.59831255486527\\
3.34240508316767 -0.642979565853876\\
3.36989661320392 -2.19855122603605\\
3.52795771556165 -0.344737281859396\\
3.60908642938297 -0.468705081104389\\
3.633767159227 -0.0869142258942954\\
3.64771292428967 -0.497706106123618\\
3.95845704333806 1.13902836386786\\
4.28787671559419 -0.370877049007023\\
4.49126911480613 0.124886576788119\\
4.66094246599578 -0.101959576394049\\
};
\addplot [
color=black,
mark size=1.8pt,
only marks,
mark=star,
mark options={solid,fill=black},
forget plot
]
table[row sep=crcr]{
2.45603243675726 2.14223371667259\\
2.47263600229526 2.30561717061105\\
2.50001407489092 3.41467918245149\\
2.51201734019518 2.88636026055057\\
2.57452967337 3.04420267891004\\
2.66389728441031 3.33443664615167\\
2.74003545321369 2.55385062151143\\
2.76134559029887 2.85282058105898\\
2.81121681946537 2.80553950462336\\
2.83572205803198 3.50613140212989\\
2.93240882151738 1.83736382860606\\
2.95132933142285 2.3278344455669\\
2.97277820250871 3.42309358452525\\
3.18645016580519 2.86323708636668\\
3.19811168025026 2.96087040305759\\
3.20054379707007 3.36426028996155\\
3.44766754218778 2.97045094070978\\
3.55044549349746 3.3423108678542\\
3.61298802756715 2.7752481026087\\
3.92649364725662 3.13464773722323\\
3.9794804007129 2.81162138452037\\
4.04509796766267 3.60837271199117\\
4.12499974343306 1.62489259889907\\
4.35808048608049 3.43174442116082\\
4.44396947553451 2.85446396177875\\
4.45773123419463 3.44001274992716\\
4.73455689780449 2.91882837111999\\
4.81734434256348 3.86936184372023\\
4.96004616443754 1.72614981742001\\
4.98667973197624 3.97646656229831\\
5.04568317333892 2.23721918395813\\
5.05981415143811 2.20753070727994\\
5.12236460770113 3.60747870867204\\
5.15237865634964 2.82130203550387\\
5.41602718861991 3.33100760724794\\
5.42865942692448 1.95127872445285\\
5.43827704576723 3.12553281220753\\
5.57622007137191 2.96826909834262\\
5.68128203534513 2.54864699512814\\
6.03404824227742 2.76085103256797\\
6.13077378268931 2.04770237298951\\
6.17929908295531 2.03049570430471\\
6.24036143751548 1.94961218002873\\
6.40477683273998 2.4129031637066\\
6.43831673674194 2.74000274893336\\
6.44553206804616 2.02409871028665\\
6.57574145554378 1.90285788942662\\
6.58398613539796 2.8180262829879\\
6.67846304010799 1.41254229685518\\
6.81474153085136 2.19611237042717\\
7.16769976074083 -0.634077314465187\\
7.24591500177456 0.666216881473127\\
7.2761859737209 1.59219088719393\\
7.82280986002478 -0.623220681132199\\
7.90474172856017 0.560531847162946\\
7.9889359716979 1.22281060405386\\
8.11648444609119 1.30419362716364\\
8.24979006935024 -0.21430144040196\\
8.67134935353465 0.21550873570508\\
9.11363291113906 1.25915837644523\\
};
\end{axis}
\end{tikzpicture}%
\end{subfigure}
\hfill
\begin{subfigure}[b]{0.45\textwidth}
\centering
\setlength\figureheight{3cm}
\setlength\figurewidth{5cm}
%
%
%
%
\begin{tikzpicture}

\begin{axis}[%
width=\figurewidth,
height=\figureheight,
ticks=none,
scale only axis,
xmin=-5,
xmax=10,
ymin=-3,
ymax=4.5
]
\addplot [
color=black,
mark size=1.2pt,
only marks,
mark=*,
mark options={solid,fill=black},
forget plot
]
table[row sep=crcr]{
-4.61354441224525 0.832361464977937\\
-4.54881155210922 2.33882718654275\\
-4.25967345950225 0.0472303796587409\\
-4.14762514910441 -0.606895785658696\\
-4.08577066491708 -0.572165580214948\\
-4.06808555702868 -0.627206745334295\\
-4.03507907733803 1.10874732208355\\
-4.02253474133478 0.511976715695957\\
-3.9342951843913 0.0540746740583324\\
-3.90117871279631 -0.198846543738207\\
-3.82523307079921 0.613628133702281\\
-3.69051836443586 0.285213162034563\\
-3.68924436738692 -0.30263416245437\\
-3.66969410761964 0.658313598849679\\
-3.66117712704449 0.50847272370766\\
-3.51176945994596 0.884055764125925\\
-3.49836008009677 0.573189789389545\\
-3.49756408290977 0.24965429488233\\
-3.49389821329317 -0.453011823665195\\
-3.46645778778749 -0.43191487028127\\
-3.44954259960669 0.274294775641414\\
-3.38822063543571 -0.428028710934555\\
-3.28169317130559 0.606384741137854\\
-3.26688872999568 -0.301550997877725\\
-3.24634463268615 -0.989576112558169\\
-3.24547085566336 -0.665737286673611\\
-3.16527852996571 -0.430819866430574\\
-3.1543548164043 -0.113540629871013\\
-3.11804108897505 -1.90445541188568\\
-2.88367283508239 -2.31342226966951\\
-2.81461036898861 -0.0820649512189811\\
-2.76489631667491 -2.36529251516757\\
-2.74844806272599 -1.69163694142652\\
-2.73017639168586 -0.120274492152151\\
-2.72511768890796 -0.872462560090612\\
-2.65823637881099 -1.35161740327127\\
-2.60162911799044 -1.48491033439179\\
-2.51486429723368 -1.88432095929566\\
-2.44967043071795 -0.686518692680404\\
-2.28980223639312 -1.84942503650549\\
-1.96109579090123 -1.11766905155629\\
-1.55984742445426 -1.80570196376781\\
-1.38058367023737 -0.81496684519942\\
-1.35110851510594 -1.72294323809533\\
-1.35040130194246 -1.88045294243807\\
-1.32476774225315 -0.94185512694711\\
-1.13702734064839 -1.99898029011068\\
-0.87569691456725 -1.64643189167766\\
-0.782094888031326 -2.38613280533269\\
-0.757045633897663 -1.80944351463797\\
-0.613573299782455 -1.79779284226345\\
-0.59736714216632 -1.27156278870325\\
-0.26702064033175 -2.01640471908406\\
-0.231462476864986 -1.40576339109543\\
-0.118578543837518 -2.10046356679899\\
0.0217493401691951 -2.42131223805036\\
0.0556254878497774 -2.67579091357577\\
0.176610229060615 -1.81939992389389\\
0.400045803621939 -2.34387699958838\\
0.409155255286757 -1.73555908452236\\
0.570170142612265 -1.95687142280684\\
0.703449672681064 -1.71329554997515\\
0.970933992683241 -1.61064556920354\\
1.01455707268897 -2.08625617196336\\
1.1474789991681 -1.98294166220577\\
1.1683686578735 -2.40328734112304\\
1.23158588820431 -2.06952648375271\\
1.29956123250541 -1.83093608204321\\
1.38340455002 -1.33640371085686\\
1.42035041816762 -1.68190192170078\\
1.44482528258127 -1.71375569317539\\
1.56634979666714 -2.78698074897775\\
1.75771628513257 -1.42499917050143\\
1.94776174921112 -0.911198419798203\\
1.97206433544139 -0.419737293592665\\
2.00811445757024 -0.636671205249424\\
2.36720378416185 -2.16218729511017\\
2.50419056777272 -1.85006833245121\\
2.5506901347904 -1.79656560017092\\
2.66104875789974 -0.923211461363686\\
2.77381224971741 -1.54055290370376\\
2.78527901429027 -2.35915787319209\\
2.79845809699079 -0.99347915854763\\
2.82710411064881 -1.42446045982968\\
2.96796484955628 -0.514714395626115\\
2.9765965705947 -0.832412983299476\\
2.98386567561932 -1.1319082875363\\
3.18707066381269 -1.59831255486527\\
3.34240508316767 -0.642979565853876\\
3.36989661320392 -2.19855122603605\\
3.52795771556165 -0.344737281859396\\
3.60908642938297 -0.468705081104389\\
3.633767159227 -0.0869142258942954\\
3.64771292428967 -0.497706106123618\\
4.28787671559419 -0.370877049007023\\
4.49126911480613 0.124886576788119\\
4.66094246599578 -0.101959576394049\\
};
\addplot [
color=black,
mark size=1.8pt,
only marks,
mark=star,
mark options={solid,fill=black},
forget plot
]
table[row sep=crcr]{
-0.792546802171743 0.927470716364932\\
-0.742456846377258 1.00064601529753\\
-0.595053446060927 0.205626055795326\\
-0.258269276567031 -0.227321988455058\\
-0.225007912239916 0.917381037592895\\
-0.2035262114635 0.472563188958718\\
-0.142836243607826 0.699012244270581\\
-0.0850627575387945 0.535823687738308\\
-0.0569059727056669 1.29229284102957\\
-0.0443805835823826 1.01356195274905\\
0.246426045365254 0.0453455468413395\\
0.407841782737313 -0.735970192237221\\
0.558051976873671 0.482009371219448\\
0.576601723681316 1.82260825280297\\
0.601367667417972 1.05067371619356\\
0.70829433216618 2.10637467493826\\
0.823815329150159 0.150926972994904\\
0.940833017176902 2.02186407760649\\
0.963345889998623 1.86468692814079\\
0.97977162590643 1.08351323187414\\
0.999773100667601 0.281582345406385\\
1.09857231360013 1.74619610246882\\
1.20259879491837 0.991333437576832\\
1.20260742815256 1.01455147754742\\
1.28936203003942 0.647415032315287\\
1.32258212082176 1.639244137448\\
1.39985712333535 2.59684196520614\\
1.45434798552229 2.80325907303559\\
1.52415651738386 2.26955173451384\\
1.53933824118 3.24329529372174\\
1.74984154025752 2.67959489804003\\
1.75180730820424 2.16217199136233\\
1.81076707621864 2.22241262325692\\
1.81228716644463 2.63009388363283\\
1.86186433660041 1.326027550149\\
1.95886691295846 2.17805845882628\\
1.96065405347892 1.68623908859185\\
2.00264474528859 1.95756479254959\\
2.12957593159012 2.05185853068469\\
2.16927105035048 2.20880088648508\\
2.22125854066778 2.46151997545208\\
2.22898457998684 3.14408526839907\\
2.45603243675726 2.14223371667259\\
2.47263600229526 2.30561717061105\\
2.50001407489092 3.41467918245149\\
2.51201734019518 2.88636026055057\\
2.57452967337 3.04420267891004\\
2.66389728441031 3.33443664615167\\
2.74003545321369 2.55385062151143\\
2.76134559029887 2.85282058105898\\
2.81121681946537 2.80553950462336\\
2.83572205803198 3.50613140212989\\
2.93240882151738 1.83736382860606\\
2.95132933142285 2.3278344455669\\
2.97277820250871 3.42309358452525\\
3.18645016580519 2.86323708636668\\
3.19811168025026 2.96087040305759\\
3.20054379707007 3.36426028996155\\
3.44766754218778 2.97045094070978\\
3.55044549349746 3.3423108678542\\
3.61298802756715 2.7752481026087\\
3.92649364725662 3.13464773722323\\
3.95845704333806 1.13902836386786\\
3.9794804007129 2.81162138452037\\
4.04509796766267 3.60837271199117\\
4.12499974343306 1.62489259889907\\
4.35808048608049 3.43174442116082\\
4.44396947553451 2.85446396177875\\
4.45773123419463 3.44001274992716\\
4.73455689780449 2.91882837111999\\
4.81734434256348 3.86936184372023\\
4.96004616443754 1.72614981742001\\
4.98667973197624 3.97646656229831\\
5.04568317333892 2.23721918395813\\
5.05981415143811 2.20753070727994\\
5.12236460770113 3.60747870867204\\
5.15237865634964 2.82130203550387\\
5.41602718861991 3.33100760724794\\
5.42865942692448 1.95127872445285\\
5.43827704576723 3.12553281220753\\
5.57622007137191 2.96826909834262\\
5.68128203534513 2.54864699512814\\
6.03404824227742 2.76085103256797\\
6.13077378268931 2.04770237298951\\
6.17929908295531 2.03049570430471\\
6.24036143751548 1.94961218002873\\
6.40477683273998 2.4129031637066\\
6.43831673674194 2.74000274893336\\
6.44553206804616 2.02409871028665\\
6.57574145554378 1.90285788942662\\
6.58398613539796 2.8180262829879\\
6.67846304010799 1.41254229685518\\
6.81474153085136 2.19611237042717\\
7.16769976074083 -0.634077314465187\\
7.24591500177456 0.666216881473127\\
7.2761859737209 1.59219088719393\\
7.82280986002478 -0.623220681132199\\
7.90474172856017 0.560531847162946\\
7.9889359716979 1.22281060405386\\
8.11648444609119 1.30419362716364\\
8.24979006935024 -0.21430144040196\\
8.67134935353465 0.21550873570508\\
9.11363291113906 1.25915837644523\\
};
\end{axis}
\end{tikzpicture}%
\end{subfigure}
\caption{
Both figures were classified using the Ginzburg-Landau functional.
The one on the left used a larger $\epsilon$ than the one on the right.
The figure shows that the smaller value of $\epsilon$ gives a much better resolution in the boundary.
}
\label{fig:Intro:Infinite:ExampleRes}
\end{figure}
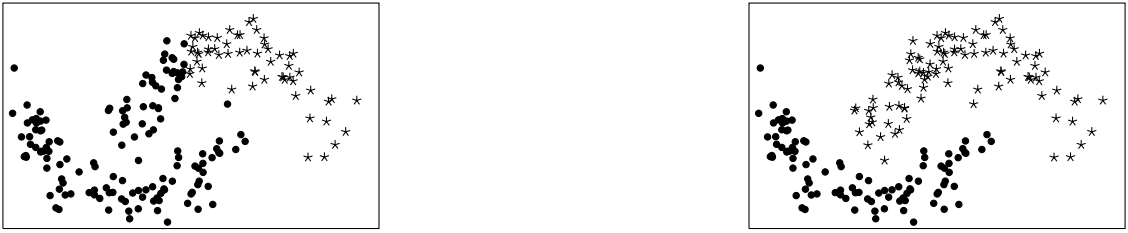

We now briefly discuss the convergence of $\mu_n\to \mu$.
Since each $\mu_n$ is defined on a different space (the domain of each $\mu_n$ is $\Psi_n$) it is not straightforward what is meant by convergence of $\mu_n\to \mu_\infty$.
By defining a map $T_n:X\to \Psi_n$ one can compare $\mu_n$ to $\mu_\infty$ by defining the piecewise constant approximation of $\mu_n$.
Formally we can say $\mu_n\to \mu_\infty$ in $TL^1$ if $\mu_n\circ T_n\to \mu_\infty$ in $L^1$, this is discussed further in Section~\ref{subsec:Prelim:Trans}.

We also include preliminary results towards characterizing the rate of convergence by considering a simple example when $\mu=\mathbb{I}_E$ for a polyhedral set $E\subset X$ and looking at the convergence in mean square:
\[ \mathbb{E} \left|\mathcal{E}_n(\mu) - \mathcal{E}_\infty(\mu)\right|^2 = \mathbb{E} \left|GTV_n(\mu) - TV(\mu;\rho,\eta)\right|^2. \]
We give an expansion of the above in terms of $\eps_n$ and $n$.
A further overview of these results is given in Section~\ref{subsec:MainRes:Rate} and the proofs in Section~\ref{sec:Rate}.

The outline of the paper is as follows.
In Section~\ref{sec:MainRes} the main result, Theorem~\ref{thm:MainRes:Cons} (the convergence of the unconstrained minimization problem), is given.
We also include an overview of the preliminary rate of convergence results to be found in Section~\ref{sec:Rate}.
Section~\ref{sec:Prelim} contains the background material and in particular: notation, a brief overview on $\Gamma$-convergence, background on total variation distances and the key details required from transportation theory.
In Section~\ref{sec:Compact} the proof of the first part of Theorem~\ref{thm:MainRes:Cons} (the compactness result) is given.
And in Section~\ref{sec:Gamma} the proof is completed with the $\Gamma$-convergence result.
Finally in Section~\ref{sec:Rate} we make the preliminary calculation regarding the rate of convergence of ``$\mathcal{E}_n\to \mathcal{E}_\infty$".

\section{Statement of Main Result and Assumptions \label{sec:MainRes}}

The assumptions on $V, \varphi$ and $\pi$ are given in the following definition.

\begin{mydef}
\label{def:MainRes:Admissible}
We say that the functions $(\rho,\epsilon_n,V,\varphi,\pi)$ where $\rho:\mathbb{R}^n\to[0,\infty)$, $V:\mathbb{R}\to [0,\infty)$, $\varphi:\mathbb{R}\to [0,\infty)$ and $\pi:\mathbb{R}^d\to \mathbb{R}$ are $\mathcal{E}_n$-admissible if the following conditions hold. 
\begin{enumerate}
\item The support of $\rho$ is $X$ where $X\subset\mathbb{R}^d$ is open, bounded, connected, Lipschitz boundary and $1\leq d <\infty$.
\item On $X$ we have that $\rho$ is a continuous probability density ($\int_X \rho(x) \, \mathrm{d} x = 1$) and bounded above and below by positive constants, i.e.
\[ 0< \inf_{x\in X} \rho(x) \leq \sup_{x\in X} \rho(x) < \infty. \]
\item $\epsilon_n^{-1} = o(f_d(n))$
where
\[ f_d(n) = \left\{ \begin{array}{ll} \sqrt{\frac{n}{\log\log n}} & \text{if } d=1 \\ \frac{\sqrt{n}}{(\log n)^{\frac{3}{2}}} & \text{if } d=2 \\ \sqrt[d]{\frac{n}{\log n}} & \text{if } d\geq 3. \end{array} \right. \]
\item The support of $\eta=\varphi\circ\pi$ is compact.
\item $\eta(0)>0$.
\item For all $\delta>0$ there exists $c_\delta,\alpha_\delta>0$ such that if $|x-y|<\delta$ then $\eta(y)\geq c_\delta \eta(\alpha_\delta x)$ and furthermore $c_\delta,\alpha_\delta\to 1$ as $\delta\to 0$.
\item $V(y)=0$ if and only if $y\in \{0,1\}$.
\item $V$ is continuous.
\item There exists $r>0$ and $\tau>0$ such that if $|t|\geq r$ then $V(t)\geq \tau |t|$.
\end{enumerate}
\end{mydef}

The lower bound on $\epsilon_n$ implies the graph with vertices $\Psi_n$ and edges weighted by $W_{ij}$ is (with probability one) eventually connected~\cite[Theorem~13.2]{penrose03}.

Note that a consequence of~4 and~6 is that $\sigma(\nu)<\infty$ however 4 also excludes any $\pi$ that is linear.
For example if $\pi(x) = w\cdot x$ then there exists $w'$ orthogonal to $w$ such that $w\cdot w' = 0$ (for $d\geq 2$).
Then $\pi(\alpha w') = 0$ for all $\alpha$ and in particular $\eta(\alpha w') = \eta(0) >0$.
Therefore the support of $\eta$ is not compact.
We discuss the linear case more in the following subsection.

Condition 6 gives the required scaling in $\varphi\circ\pi$.
If $T_n:X\to \Psi_n$ is a map such that $T_{n\#}P = P_n$ then the continuous approximation of the weights $W_{ij}$ reads as
\[ W_{ij} = \frac{1}{\epsilon_n^d} \varphi\circ\pi\left(\frac{\xi_i-\xi_j}{\epsilon_n}\right) \approx \frac{1}{\epsilon_n^d} \int_{T_n^{-1}(\xi_i)} \int_{T_n^{-1}(\xi_j)} \varphi\circ\pi\left(\frac{x-y}{\epsilon_n}\right) \; \mathrm{d} x \; \mathrm{d} y. \]
Condition 6 is sufficient to formalise this reasoning.
We give two different classes of functions in Proposition~\ref{prop:MainRes:SatAss} below that satisfy the assumptions.
The first is a subset of isotropic functions that we extend in Corollary~\ref{cor:MainRes:Isotropic}, the second is a set of indicator functions.

\begin{proposition}
\label{prop:MainRes:SatAss}
Assume that either
\begin{enumerate}
\item $\varphi:[0,\infty)\to [0,\infty)$ is decreasing, $\varphi(0)>0$, is Lipschitz with compact support and $\pi:\mathbb{R}^d\to [0,\infty)$ is defined by $\pi(x)=|x|$, or
\item $\varphi(0)>0$, $\varphi(1)=0$, $\pi:\mathbb{R}^d\to \{0,1\}$ is given by $\pi(x)=\mathbb{I}_{E^c}(x)$ for an open, bounded and convex set $E\subset\mathbb{R}^d$ with $0\in E$.
\end{enumerate}
Then $(\varphi,\pi)$ are $\mathcal{E}_n$-admissible.
\end{proposition}

We prove the above proposition in Appendix~\ref{app:ProofPropMainResSatAss}.
We now state the main result.
The key idea is that optimal transportation theory provides a natural extension of $\mu_n:\Psi_n\to \mathbb{R}$ to $\tilde{\mu}_n:X\to \mathbb{R}$ for which we can use to define the convergence $\mu_n\to \mu$.
This is formalised as convergence in $TL^1$, see Section~\ref{subsec:Prelim:Trans}.
Establishing $\Gamma$-convergence of functionals, see Definition~\ref{def:Prelim:Gamma:Gamcon}, and the compactness of minimizers leads to the convergence of minimizers as in Theorem~\ref{thm:Prelim:Gamma:Conmin}.
We use the notation $L^1(\Psi_n)$ as convenient notation for functions defined on~$\Psi_n$.
We define the space $B(X;\rho,\eta)$, the space of functions of bounded variation with respect to a measure $\rho$ and a interaction potential $\eta$, in Definition~\ref{def:Prelim:TV:TV}.

\begin{theorem}
\label{thm:MainRes:Cons}
Let $\mathcal{E}_n:L^1(\Psi_n)\to[0,\infty)$ and $\mathcal{E}_\infty:L^1(X)\to [0,\infty]$ be defined by~\eqref{eq:Intro:Finite:En} and~\eqref{eq:Intro:Infinite:EinftySurface} respectively.
If $(\rho,\epsilon_n,V,\varphi,\pi)$ are $\mathcal{E}_n$-admissible (in the sense of Definition~\ref{def:MainRes:Admissible}) and $\xi_i\iid \rho$
then, with probability one, the following hold
\begin{enumerate}
\item Compactness: Let $\mu_n$ be a sequence of functions on $\Psi_n$ such that $\sup_{n\in\mathbb{N}} \mathcal{E}_n(\mu_n)<\infty$ then $\mu_n$ is relatively compact in $TL^1$ and each cluster point is in $BV(X;\rho,\varphi\circ\pi)\cap L^1(X;\{0,1\})$.
\item $\Gamma$-limit: We have
\[ \Glim_{n\to \infty} \mathcal{E}_n = \mathcal{E}_\infty. \]
\end{enumerate}
\end{theorem}


The compactness result is proved in Section~\ref{sec:Compact} and the $\Gamma$-convergence result in Section~\ref{sec:Gamma}.

Proposition~\ref{prop:MainRes:SatAss} allows one to apply Theorem~\ref{thm:MainRes:Cons} to isotropic weights $W_{ij}$ when $\varphi$ is decreasing, $\varphi(0)>0$ and Lipschitz.
We now show that the Lipschitz assumption can be removed.

\begin{corollary}
\label{cor:MainRes:Isotropic}
Let $\mathcal{E}_n:L^1(\Psi_n)\to[0,\infty)$ and $\mathcal{E}_\infty:L^1(X)\to [0,\infty]$ be defined by~\eqref{eq:Intro:Finite:En} and~\eqref{eq:Intro:Infinite:EinftySurface} respectively.
If
\begin{enumerate}
\item[$\bullet$] $\rho$ satisfies conditions 1-3 in Definition~\ref{def:MainRes:Admissible}, 
\item[$\bullet$] $\xi_i\iid \rho$,
\item[$\bullet$] $\varphi:[0,\infty)\to [0,\infty)$ is decreasing, compactly supported, $\varphi(0)>0$ and continuous at 0,
\item[$\bullet$] $\pi(x)=|x|$, and
\item[$\bullet$] $V$ satisfies conditions 11-13 in Definition~\ref{def:MainRes:Admissible}
\end{enumerate}
then, with probability one, the conclusions of Theorem~\ref{thm:MainRes:Cons} hold.
\end{corollary}

We prove the corollary in Appendix~\ref{app:ProofCorMainResIsotropic}

\subsection{Convergence of the Ginzburg-Landau Functional for Linear Feature Projections \label{subsec:MainRes:Linear}}

After Definition~\ref{def:MainRes:Admissible} we discussed how the compact support assumption on $\eta=\varphi\circ \pi$ did not allow for linear $\pi$.
One case that is of interest, that is not covered by Theorem~\ref{thm:MainRes:Cons} or Corollary~\ref{cor:MainRes:Isotropic} is $\pi(x) = x_i$ that corresponds to weighting the graph based on differences in one direction only.

This case is of particular interest for high dimensional data.
If the data comes from a high, potentially infinite, dimensional space then it becomes necessary to identify a finite number of principal dimensions upon which to define the edge weights.
Isotropic weights are unrealistic in high dimensions and infeasible in infinite dimensions due to the lack of integrability of $\eta$.
This motivates our study of linear $\pi$, which can include $\pi(x) = \sum_{i\in \mathcal{K}} x_i$ for a finite set $\mathcal{K}$.
Although we do not consider infinite dimensional data spaces here we believe the results of this section can be extended from the finite dimensional setting, albeit with a modified limit $\hat{\mathcal{E}}_\infty$ to the one we define in~\eqref{eq:MainRes:Linear:hatEinfty}.
With a more thorough treatment one should also be able to extend the convergence results of this subsection (i.e. to infinite dimensional spaces) to non-linear $\pi$.

The underlying problem, and why we should not expect a linear choice of $\pi$ to imply that $\mathcal{E}_\infty$ is well defined, is that it becomes harder to ensure that $\sigma(\nu)<\infty$.
In particular, we expect $\mathcal{E}_\infty\equiv +\infty$.
In many applications we anticipate that only a small number of dimensions are relevant and hence we, in this section, consider the classifier as a functional on a projected data space.
We define $\hat{\Psi}_n = \{\hat{\xi}_i\}_{i=1}^n \subset \mathbb{R}$ by $\hat{\xi}_i = \pi(\xi_i)$ and $\hat{\mu}_n:\hat{\Psi}_n\to[0,\infty)$.
We now consider the energy
\begin{equation} \label{eq:MainRes:Linear:hatEn}
\hat{\mathcal{E}}_n(\hat{\mu}_n) = \frac{1}{\epsilon_n} \frac{1}{n} \sum_{i=1}^n V(\hat{\mu}_n(\hat{\xi}_i)) + \frac{1}{\epsilon_n} \frac{1}{n^2} \sum_{i,j} \hat{W}_{ij} |\hat{\mu}_n(\hat{\xi}_i)-\hat{\mu}_n(\hat{\xi}_j)|
\end{equation}
where
\[ \hat{W}_{ij} = \frac{1}{\epsilon_n} \varphi\circ\pi\left(\frac{\xi_i-\xi_j}{\epsilon_n}\right) = \frac{1}{\epsilon_n} \varphi\left(\frac{\hat{\xi}_i-\hat{\xi}_j}{\epsilon_n}\right). \]
We point out that the scaling is with respect to $\epsilon_n$ rather than $\epsilon_n^d$ as in~\eqref{eq:Intro:Finite:W}.
In this case the $\Gamma$-limit $\hat{\mathcal{E}}_\infty:L^1(\mathbb{R})\to \mathbb[0,\infty]$ is given by
\begin{equation} \label{eq:MainRes:Linear:hatEinfty}
\hat{\mathcal{E}}_\infty(\hat{\mu}) = \hat{\sigma} \sum_{x\in\partial\{\hat{\mu}=1\}} \left( \int_{\pi^{-1}(x)} \rho(y) \; \mathrm{d} y \right)^2
\end{equation}
where
\[ \hat{\sigma} = \int_{\mathbb{R}} \varphi(x)|x| \; \mathrm{d} x. \]

The analogous assumptions are given in the definition below.

\begin{mydef}
\label{def:MainRes:Linear:Admissible}
We say that the functions $(\rho,\epsilon_n,V,\varphi,\pi)$ where $\rho:\mathbb{R}^d\to[0,\infty)$, $V:\mathbb{R}\to [0,\infty)$, $\varphi:\mathbb{R}\to [0,\infty)$ and $\pi:\mathbb{R}^d\to \mathbb{R}$ are $\hat{\mathcal{E}}_n$-admissible if the following hold.
\begin{enumerate}
\item The support of $\rho$ is $X$ where $X\subset\mathbb{R}^d$ and the support of $\pi_{\#} \rho$ is open, bounded and connected.
\item On $X$ we have that $\rho$ is a continuous probability density, bounded above and $\inf_{x\in X} \pi_{\#}\rho(x)>0$.
\item $\epsilon_n^{-1} = o(f_1(n))$ where $f_1$ is given in Definition~\ref{def:MainRes:Admissible}.
\item $\pi$ is linear.
\item The support of $\varphi$ is compact.
\item $\varphi(0)>0$.
\item For all $\delta>0$ there exists $c_\delta,\alpha_\delta>0$ such that if $|x-y|<\delta$ then $\varphi(y)\geq c_\delta \varphi(\alpha_\delta x)$ and furthermore $c_\delta,\alpha_\delta\to 1$ as $\delta\to 0$.
\item $V(y)=0$ if and only if $y\in \{0,1\}$.
\item $V$ is continuous.
\item There exists $r>0$ and $\tau>0$ such that if $|t|\geq r$ then $V(t)\geq \tau |t|$.
\end{enumerate}
\end{mydef}

The convergence theorem is given below.

\begin{theorem}
\label{thm:MainRes:Linear:Conshat}
Let $\hat{\Psi}_n=\{\hat{\xi}_i\}_{i=1}^n$ where $\hat{\xi}_i=\pi(\xi_i)$.
Define $\hat{\mathcal{E}}_n:L^1(\hat{\Psi}_n)\to[0,\infty)$ and $\hat{\mathcal{E}}_\infty:L^1(X)\to [0,\infty]$ by~\eqref{eq:MainRes:Linear:hatEn} and~\eqref{eq:MainRes:Linear:hatEinfty} respectively.
If $(\rho,\epsilon_n,V,\varphi)$ are $\hat{\mathcal{E}}_n$-admissible (in the sense of Definition~\ref{def:MainRes:Linear:Admissible}) and $\xi_i\iid \rho$ then, with probability one, the following hold
\begin{enumerate}
\item Compactness: Let $\hat{\mu}_n$ be a sequence of functions on $\hat{\Psi}_n$ such that $\sup_{n\in\mathbb{N}} \hat{\mathcal{E}}_n(\hat{\mu}_n)<\infty$ then $\hat{\mu}_n$ is relatively compact in $TL^1$ and each cluster point is in $BV(\mathbb{R};\pi_{\#}\rho,\eta)\cap L^1(\mathbb{R};\{0,1\})$.
\item $\Gamma$-limit: we have
\[ \Glim_{n\to \infty} \mathcal{\hat{E}}_n = \mathcal{\hat{E}}_\infty. \]
\end{enumerate}
\end{theorem}

The proof is an application of Theorem~\ref{thm:MainRes:Cons} to the 1-dimensional data set $\hat{\Psi}_n$.

\subsection{Comments on the Main Result \label{subsec:MainRes:Comments}}

The classical Ginzburg-Landau functional:
\[ F_\epsilon(\mu) = \epsilon\int_X |\nabla \mu(x)|^2 \; \mathrm{d} x + \frac{1}{\epsilon} \int_X V(\mu(x)) \; \mathrm{d} x \]
has been well studied and its convergence to a total variation functional
\[ F_\infty(\mu) = \sigma_V \int_X |\nabla \mu(x)| \; \mathrm{d} x \]
known for some time~\cite{modica87} and similar results for the anisotropic version~\cite{alberti98}.
More recent results have studied this functional on a (deterministic) regular graph.
In~\cite{vangennip12a} the authors show the $\Gamma$-convergence and compactness of two variants of the Ginzburg-Landau functional where $\{\xi_i\}_{i=1}^n\subset\mathbb{R}^2$ form a 4-regular graph.
Let us exploit the structure of the graph by writing data as $\{\xi_{i,j}\}_{i,j=1}^n$ where $\xi_{i,j}$, $\xi_{i,j+1}$ are neighbours, as are $\xi_{i,j}$ and $\xi_{j+1,i}$.
The two variants of the Ginzburg-Landau functional considered in~\cite{vangennip12a} are
\begin{align*}
h_{n,\epsilon}(\mu) & = \frac{1}{\epsilon} \sum_{i,j=1}^n V(\mu(\xi_{i,j})) + \frac{1}{n} \sum_{i,j=1}^n \left( \left| \mu(\xi_{i+1,j}) - \mu(\xi_{i,j}) \right|^2 + \left| \mu(\xi_{i,j+1}) - \mu(\xi_{i,j}) \right|^2 \right) \\
k_{n,\epsilon}(\mu) & = \frac{1}{\epsilon n^2} \sum_{i,j=1}^n V(\mu(\xi_{i,j})) + \epsilon \sum_{i,j=1}^n \left( \left| \mu(\xi_{i+1,j}) - \mu(\xi_{i,j}) \right|^2 + \left| \mu(\xi_{i,j+1}) - \mu(\xi_{i,j}) \right|^2 \right).
\end{align*}
The first functional $h_{n,\epsilon}$ $\Gamma$-converges as $\epsilon\to 0$ (for a fixed $n$) to a total variation function $h_{n,0}$ in a discrete setting defined by
\[ h_{n,0}(\mu) = \left\{ \begin{array}{ll} \frac{1}{n} \sum_{i,j=1}^n \left( \left| \mu(\xi_{i+1,j}) - \mu(\xi_{i,j}) \right|^2 + \left| \mu(\xi_{i,j+1}) - \mu(\xi_{i,j}) \right|^2 \right) & \text{if } \mu \in L^1(\Psi_n;\{0,1\}) \\ \infty & \text{otherwise.} \end{array} \right. \]
As $\epsilon\to 0$ and $n\to \infty$ sequentially or for $\epsilon = n^{-\alpha}$ for $\alpha$ within some range then $h_{n,\epsilon}$ $\Gamma$-converges to an anisotropic total variation in a continuous setting
\[ \Glim_{\substack{n\to \infty\\\epsilon\to 0}} h_{n,\epsilon} = \int_{\mathbb{T}^2} \left| \frac{\partial \mu}{\partial x} \right| + \left| \frac{\partial \mu}{\partial y} \right| \]
and $k_{n,\epsilon}$ $\Gamma$-converges to an isotropic total variation
\[ \Glim_{\substack{n\to \infty\\\epsilon\to 0}} k_{n,\epsilon} = \int_{\mathbb{T}^2} \left| \nabla \mu \right| \]
upto renormalization.
Also discussed in~\cite{vangennip12a} is the application to the constrained minimization problem


In the remainder of the paper we will work on point clouds, that are random graphs, rather than deterministic regular graphs.

\paragraph*{Convergence of the Graph Total Variation.}
The Ginzburg-Landau functional consists of two terms, the first is a regularization on the derivative and the second is a a penalization on states outside of $\{0,1\}$.
In this paper we use the graph total variation on graphs~\eqref{eq:Intro:Finite:TVG} as the regularization on the derivative.
In more generality one can define the $p$-Laplacian on graphs~\cite{zhou05} by
\[ J_p(\mu) = \frac{1}{\eps_n^p} \frac{1}{n^2} \sum_{i,j} W_{ij} |\mu(\xi_i) - \mu(\xi_j)|^p. \]
The graph total variation corresponds to $p=1$ and the results of~\cite{vangennip12a} described above corresponded to $p=2$.


An interesting and important question in its own right is the convergence of the $p$-Laplacian~\cite{elalaoui16,buhler09,ando07,zhou11,bridle13,zhou11a}.
For isotropic weights, i.e. $\pi(x)=|x|$, the following result was established for $p=1$ in~\cite{trillos15}.
\begin{enumerate}
\item Compactness: Let $\mu_n$ be any sequence of functions on $\Psi_n=\{\xi_i\}$, where $\xi_i\iid P$, such that $GTV_n(\mu_n)$ is bounded and $\mu_n$ is bounded in $TL^1$ then $\mu_n$ is almost surely relatively compact in $TL^1$ and each cluster point is in $BV(X;\rho,\eta)$.
\item $\Gamma$-limit: we have, with probability one,
\[ \Glim_{n\to \infty} GTV_n = TV(\cdot;\rho,\eta). \]
\end{enumerate}
A consequence of our proofs is that this result is also true for any $X,\rho,\eps_n,\varphi$ and $\pi$ satisfying Assumptions 1-6 in Definition~\ref{def:MainRes:Admissible}.
In particular, the $\Gamma$-convergence holds analogously to Theorem~\ref{thm:Gamma:EntoEinfty} and the compactness can be reduced, as in the proof of Proposition~\ref{prop:Compact:ComSecondTermData}, to the isotropic case where the results of~\cite{trillos15} apply.
This generalises the result of~\cite{trillos15} to anisotropic weights.


\paragraph*{Convergence of minimizers.}
The results of the Theorem~\ref{thm:MainRes:Cons} can be understood as implying the convergence of minimizers in the following sense.
For a sequence of closed sets $\Theta_n\subseteq L^1(\Psi_n)$ and $\Theta\subseteq L^1(X)$ which we assume respect the $\Gamma$-convergence, that is the following hold: (1) if $\zeta\in \Theta$ then $\zeta_n:= \zeta\big|_{\Psi_n}\in \Theta_n$, (2) there exists $\zeta\in \Theta$ such that $\mathcal{E}_\infty(\zeta)<\infty$ and (3) any sequence $\zeta_n$ with $\zeta_n\to \zeta$ implies $\zeta\in \Theta$.
With probability one the following statements hold.
\begin{enumerate}
\item Convergence of the minimum: $\lim_{n\to \infty}\inf_{\Theta_n} \mathcal{E}_n = \min_{\Theta} \mathcal{E}_\infty$.
\item Convergence of minimizers: if $\mu_n\in L^1(\Psi_n)$ are a sequence satisfying
\[ \mathcal{E}_n(\mu_n) \leq \min_\mu \mathcal{E}_n(\mu) + \delta_n \]
for a sequence $\delta_n\to^+ 0$ (we call $\mu_n$ a minimizing sequence) then this sequence is precompact in $TL^1$ and furthermore any cluster point minimizes $\mathcal{E}_\infty$.
\end{enumerate}
The proof is a simple consequence of Theorem~\ref{thm:MainRes:Cons} and Theorem~\ref{thm:Prelim:Gamma:Conmin}.

Alternatively one could let $g_n:L^1(\Psi_n)\to [0,\infty)$ be a sequence that continuously converges to $g:L^1(X)\to [0,\infty)$, i.e. $g_n(\zeta_n)\to g(\zeta)$ whenever $\zeta_n\to \zeta$ in $L^1$, and then since the $\Gamma$-convergence is stable under continuous perturbations the results of this paper imply (with probability one) that
\begin{enumerate}
\item $\lim_{n\to \infty}\inf_{L^1(\Psi_n)} \left( \mathcal{E}_n + g_n \right) = \min_{L^1(X)} \left( \mathcal{E}_\infty + g \right)$, and
\item if $\mu_n\in L^1(\Psi_n)$ are a minimizing sequence for $\mathcal{E}_n + g_n$ then this sequence is precompact in $TL^1$ and furthermore any cluster point minimizes $\mathcal{E}_\infty + g$.
\end{enumerate}
For example one could use this in order to fit data, e.g.
\[ g_n(\mu;\zeta) = \lambda \int_X \left| \mu(T_n(x)) - \zeta(x) \right| \; \mathrm{d} x \]
where $\zeta$ is a known function (data) and $T_n$ is a sequence of stagnating transport maps (a sequence such that $\|\mathrm{Id} - T_n\|_{L^p}\to 0$, see Section~\ref{subsec:Prelim:Trans}).
In this case $g(\mu) = \lambda \int_X \left| \mu(x) - \zeta(x) \right| \; \mathrm{d} x$.

\paragraph*{Choice of scaling.}
The natural choice of scaling in $\mathcal{E}_n$ between the two terms is not a-priori obvious.
One could write
\[ \mathcal{E}_n(\mu) = \frac{1}{\gamma_n} \sum_{i=1}^n V(\mu(\xi_i)) + \frac{1}{\epsilon_n} \frac{1}{n^2} \sum_{i,j} W_{ij} |\mu(\xi_i)-\mu(\xi_j)|. \]
The proof of Theorem~\ref{prop:Compact:ComSecondTermData} requires $\frac{\gamma_n}{\epsilon_n} = O(1)$.
One can show that Theorem~\ref{thm:MainRes:Cons} holds for $\gamma_n=O(\epsilon_n)$.
For simplicity it is assumed that $\gamma_n=\epsilon_n$.

\paragraph*{Generalization to $L^p$ spaces for $p>1$.}
The literature on the related problem in $L^p$ spaces has been well developed.
Define
\[ \mathcal{E}_n(\mu) = \frac{1}{\epsilon_n} \sum_{i=1}^n V(\mu(\xi_i)) + \frac{1}{\epsilon_n^p} \frac{1}{n^2} \sum_{i,j} W_{ij} \left| \mu(\xi_i) - \mu(\xi_j) \right|^p. \]
Then one can show the results of Theorem~\ref{thm:MainRes:Cons} still hold for the same limiting energy $\mathcal{E}_\infty$.
In particular the convergence of minimizing sequences is still in $TL^1$.

\paragraph*{Size of the phase transition.}
Although we don't formally state the result, it is well known that for the Ginzburg-Landau functional the phase transition is of order $\epsilon_n$, see for example~\cite[Theorem 6.4]{braides02}.
By the proof of Lemma~\ref{lem:Gamma:EntoEinftylimsup} the same is true in the setting described in this paper.
That is if $\mu_n$ is a recovery sequence (see Definition~\ref{def:Prelim:Gamma:Gamcon}) for $\mathcal{E}_n$ then $\frac{1}{n}\#\{\xi\in \Psi_n \, : \, \mu_n(\xi)\in (c,1-c) \} = O(\epsilon_n)$.

\subsection{Preliminary Results on the Rate of Convergence \label{subsec:MainRes:Rate}}

We include some preliminary results concerning the rate of convergence for $\inf \mathcal{E}_n \to \min \mathcal{E}_\infty$.
The problem is simplified by looking at the convergence $GTV_n(\mu) \to TV(\mu;\rho,\eta)$ for $\mu=\mathbb{I}_E$ where $E$ is a polyhedral set.
To characterize the rate of convergence we look for convergence in mean square.
It is shown in Theorem~\ref{thm:Rate:ConvMeanSq} that
\[ \mathbb{E} \left| \mathcal{E}_n(\mu) - \mathcal{E}_\infty(\mu) \right|^2 = \mathbb{E} \left| GTV_n(\mu) - TV(\mu;\rho,\eta) \right|^2 = O(\eps_n) + \frac{\kappa_1}{n\eps_n} + \frac{\kappa_2}{n^2\eps_n^{d+1}} + O\left(\epsilon_n^2\right) \]
as $n\to \infty$ for a constants $\kappa_i$ given in Section~\ref{sec:Rate}.
Even though (by Lemma~\ref{lem:Gamma:EntoEinftylimsup}) one has $GTV_n(\mu) \to TV(\mu;\rho,\eta)$ almost surely convergence in expectation does not immediately follow, this is shown in Theorem~\ref{thm:Rate:ConvMean}.
The leading $O(\epsilon_n)$ term above corresponds to approximating $\mu$ along edges of $E$ where an edge is the intersection of two faces of $E$ (see Section~\ref{sec:Rate} for a precise explanation of the notation we use to describe polyhedral sets).
The error in the edges causes a bias in the estimate.
For example if one considers the function $\mu = \mathbb{I}_{H\cap X}$ where $H=\{ x \, : \, w\cdot x >0\}$ for some $w$ is any half space then $\mu$ is a polyhedral function with no edges in $X$ and one can show $\mathbb{E}GTV_n(\mu) = TV(\mu;\rho,\eta)$.
It follows from our proofs that $\mathbb{E} \left| GTV_n(\mu) - TV(\mu;\rho,\eta) \right|^2 = \frac{\kappa_1}{n\eps_n} + \frac{\kappa_2}{n^2\eps_n^{d+1}} + O\left(\frac{1}{n}\right)$.

\section{Preliminary Material \label{sec:Prelim}}

\subsection{Notation \label{subsec:Prelim:Not}}

The space of functions from $Z$ onto $Y$ that are $L^p$-integrable are denoted by $L^p(Z;Y)$ (for $1\leq p\leq \infty$).
Usually either $Y=\{0,1\}$ or $Y=\mathbb{R}$.
If $Y=\mathbb{R}$ then we write $L^p(Z)$ instead of $L^p(Z;\mathbb{R})$.
When we use the $L^p$ norm with respect to a measure $P$ the $Y$ dependence is suppressed and we write $L^p(X;P)$.
It will be obvious from the context what is meant.

The Euclidean norm is given by $|\cdot|$ and with a small abuse of notation the dimension is inferred from the argument.
The ball centred at $x$ and with radius $r$ in $\mathbb{R}^d$ is given as
\[ B(x,r) = \left\{ y\in\mathbb{R}^d : |x-y|<r \right\}. \]
When the ball is centred at the origin we write $B(0,r)$.

\subsection{\texorpdfstring{$\Gamma$}{Gamma}-Convergence \label{subsec:Prelim:Gamma}}

$\Gamma$-convergence was introduced in the 1970's by De Giorgi as a tool for studying sequences of variational problems.
A key feature is the natural emergence of singular objects such as characteristic functions which we will interpret as classifiers.
We will adopt the view point that $\Gamma$-convergence can used as a data analysis tool as it provides direct links between statistical modelling assumptions and the corresponding minimizers, cf. e.g. Section~\ref{subsec:Intro:Ex}.
Very accessible introductions to $\Gamma$ convergence can be found in ~\cite{braides02,dalmaso93}.

We have the following definition of $\Gamma$-convergence.

\begin{mydef}[$\Gamma$-convergence]
\label{def:Prelim:Gamma:Gamcon}
Let $(A,d_A)$ be a metric space.
A sequence $f_n :A\to \mathbb{R}\cup\{\pm\infty\}$ is said to \textit{$\Gamma$-converge} on the domain $A$ to $f_\infty :A\to \mathbb{R}\cup\{\pm\infty\}$ with respect to $d_A$, and write $f_\infty = \Glim_n f_n$, if for all $\zeta\in A$:
\begin{itemize}
\item[(i)] (lim inf inequality) for every sequence $(\zeta_n)$ converging to $\zeta$
\[ f_\infty(\zeta) \leq \liminf_{n\to \infty} f_n(\zeta_n); \]
\item[(ii)] (recovery sequence) there exists a sequence $(\zeta_n)$ converging to $\zeta$ such that
\[ f_\infty(\zeta) \geq \limsup_{n\to \infty} f_n(\zeta_n). \]
\end{itemize}
\end{mydef}

When it exists the $\Gamma$-limit is always lower semi-continuous, and hence there exists minimizers over compact sets.
The following result justifies the use of $\Gamma$-convergence as a variational type of convergence.

\begin{theorem}[Convergence of Minimizers]
\label{thm:Prelim:Gamma:Conmin}
Let $(A,d_A)$ be a metric space and $f_n: A\to [0,\infty]$ be a sequence of functionals.
Let $\mu_n$ be a minimizing sequence for $f_n$.
If $\mu_n$ are precompact and $f_\infty = \Glim_n f_n$ where $f_\infty:A\to[0,\infty]$ is not identically $+\infty$ then
\[ \min_{A} f_\infty = \lim_{n\to \infty} \inf_{A} f_n. \]
Furthermore any cluster point of $\mu_n$ minimizes $f_\infty$.
\end{theorem}

A simple consequence of the above is if one can show that the $\Gamma$-limit has a unique minimizer then any minimizing sequence converges (without the recourse to subsequences).

\subsection{Total Variation Distance \label{subsec:Prelim:TV}}

The energy $\mathcal{E}_\infty$ can also be written as a BV norm:
\begin{equation} \label{eq:Prelim:TV:Einfty}
\mathcal{E}_\infty(\mu) = \left\{ \begin{array}{ll}  TV(\mu;\rho,\eta) & \text{if } \mu \in L^1(X;\{0,1\}) \\ \infty & \text{otherwise} \end{array} \right.
\end{equation}
where
\begin{align}
TV(\mu;\rho, \eta) & = \sup \Bigg\{ \int_X \mu(x)\, \text{div} \left(\phi(x)\right) \; \mathrm{d} x\;: \; \phi\in C^\infty_c(X;\mathbb{R}^d),\; \sup_{x \in X}\sigma^*\left(-\rho^{-2}(x)\phi(x)\,\right)< \infty  \Bigg\}, \label{eq:Prelim:TV:TV} \\
\sigma^*(\phi) & = \sup\left\{ \nu\cdot \phi -\sigma(\nu) \;:\; \nu \in \mathbb{R}^d\right\} \in \{0,\infty\},  \label{eq:Prelim:TV:sigmastar} \\
\sigma(\nu) & = \int_{\mathbb{R}^d} \eta(x) |x\cdot \nu| \; \mathrm{d} x, \label{eq:Prelim:TV:sigma}
\end{align}
$\eta=\varphi\circ\pi$ and $\rho$ is a probability density on $X$.
The surface energy density $\sigma$ is 1-homogeneous and the Legendre transform $\sigma^*$ is a characteristic function which assumes the values 0 or $\infty$.
The space of functions with bounded variation is defined below.

\begin{mydef}
\label{def:Prelim:TV:TV}
For a domain $X\subset \mathbb{R}^d$ the weighted total variation $TV(\cdot;\rho,\eta)$ of function $\mu\in L^1(X)$ with respect to a density $\rho$ and potential $\eta$ is defined by~\textnormal{(\ref{eq:Prelim:TV:TV}-\ref{eq:Prelim:TV:sigma})}.
The space of functions with finite weighted total variation is denoted by $BV(X;\rho,\eta)$.
The standard total variation distance on $X$ is defined by
\[ \widehat{TV}(\mu) = \sup \left\{ \int_X \mu(x) \, \text{div}(\phi) \; \mathrm{d} x \; : \; \phi\in C_c^\infty(X), \|\phi\|_{L^\infty(X)} \leq 1 \right\}. \]
The standard bounded variation space $\widehat{BV}(X)$ is the set of functions such that $\widehat{TV}(\mu)<\infty$.
\end{mydef}

The equivalence of definitions~\eqref{eq:Intro:Infinite:EinftySurface} and~\eqref{eq:Prelim:TV:Einfty} can be seen from the simplification of $TV(\cdot;\rho,\eta)$ when $\mu\in C^1$:
\[ TV(\mu;\rho,\eta) = \int_X \sigma(\nabla \mu(x)) \rho^2(x) \; \mathrm{d} x = \int_X \int_{\mathbb{R}^d} \eta(y) |y\cdot \nabla \mu(x)| \rho^2(x) \; \mathrm{d} y \; \mathrm{d} x. \]
One may also write
\[ TV(\mu;\rho,\eta) = \int_{\mathbb{R}^d} \eta(z) TV_z(\mu;\rho) \; \mathrm{d} z \]
where $TV_z(\cdot;\rho)$ is defined by
\[ TV_z(\mu;\rho) = \sup\left\{ \int_X \mu(x) \; \text{div}(\phi(x)) \; \mathrm{d} x \, : \, \phi\in C_c^\infty(X;\mathbb{R}^d), -\nu \cdot \phi(x) \leq |z \cdot \nu | \rho^2(x) \, \forall \nu,x \in \mathbb{R}^d \right\} \]

The following proposition is a slight generalization of a well known result regarding the convergence of difference quotients to the total variation semi-norm.
The proof is omitted but it is a trivial adaptation of, for example,~\cite[Theorem 13.48]{leoni09}.

\begin{proposition}
\label{prop:Prelim:TV:Liminf}
Assume $\mu_n\to \mu$ in $L^1$.
For a sequence $\epsilon_n \to 0$ and a function $\rho:X\to [0,\infty)$ define $f_n : X\to [0,\infty)$ by
\[ f_n(z) = \frac{1}{\epsilon_n} \int_X \left|\mu_n(x+\epsilon_n z) - \mu_n(x) \right| \rho^2(x) \; \mathrm{d} x. \]
Then
\[ \liminf_{n\to \infty} f_n(z) \geq TV_z(\mu;\rho). \]
\end{proposition}

For each $\mu\in BV(X;\rho,\eta)$ the following theorem gives the existence of a measure that one can understand as the weak derivative of $\mu$.
See for example~\cite{attouch06,evans92} for more details.

\begin{theorem}
\label{thm:Prelim:TV:Measure}
For every $\mu\in BV(X;\rho,\eta)$ there exists a Radon measure $\lambda_{\rho,\eta}$ on $X$ and a $\lambda_{\rho,\eta}$-measurable function $\alpha:X\to \mathbb{R}$ such that $\alpha(x)=1$ for $\lambda_{\rho,\eta}$-almost every $x\in X$ and
\[ \int_X \mu(x) \mathrm{div} \phi(x) \; \mathrm{d} x = - \int_X \frac{\phi(x) \cdot x}{\rho^2(x)\sigma(x)}  \alpha(x) \; \lambda_{\rho,\eta}(\mathrm{d}x) \]
for all $\phi\in C_c^1(X;\mathbb{R}^d)$.
In particular,
\[ \lambda_{\rho,\eta}(X) = TV(\mu;\rho,\eta). \]
For the standard total variation distance we write $\hat{\lambda}$ and have the following relationship:
\[ \lambda_{\rho,\eta}(\mathrm{d} x) = \rho^2(x) \sigma(x) \; \hat{\lambda}(\mathrm{d} x). \]
In particular
\begin{equation} \label{eq:Prelim:TV:TVIntrho}
TV(\mu;\rho,\eta) = \int_X \rho^2(x) \sigma(x) \; \hat{\lambda}(\mathrm{d} x).
\end{equation}
\end{theorem}

Using~\eqref{eq:Prelim:TV:TVIntrho} we first prove Theorem~\ref{thm:MainRes:Cons} (in particular the $\Gamma$-convergence statement) for Lipschitz $\rho$ then generalize to continuous $\rho$ by taking a monotonic sequence of Lipschitz functions $\rho_k\to \rho$ and applying the monotone convergence theorem. 

A useful approximation result we will make use of is for all $\mu\in BV(X;\rho,\eta)$ there exists a sequence $\mu_n\in BV(X;\rho,\eta)\cap C^\infty(\mathbb{R}^d)$ such that $\mu_n\to \mu$ in $L^1(X)$ and $TV(\mu_n;\rho,\eta)\to TV(\mu;\rho,\eta)$ or equivalently $\lambda^{(\rho,\eta)}_n(X)\to \lambda^{(\rho,\eta)}(X)$ (where $\lambda_n^{(\rho,\eta)}$ is the measure given by Theorem~\ref{thm:Prelim:TV:Measure} and induced by $\mu_n$), see e.g.~\cite[Theorem 2, Section 5.2.2]{evans92}.

\subsection{Transportation Theory \label{subsec:Prelim:Trans}}

This subsection contains the preliminary material required to compare functions defined on different domains.
The key idea in~\cite{trillos15} was to use transportation maps in order to extend functions to a common domain.
This allows one to define an $L^p$-type convergence for functions on different domains.

Let us consider a sequence of functions $\mu_n$ on discrete domains $\Psi_n$ and a function $\nu$ on the continuous domain $X$.
In order to compare $\mu_n$ to $\nu$ we first use a piecewise constant approximation $\tilde{\mu}_n$ to extend $\mu_n$ onto the space $X$.
For a map $T_n:X\to \Psi_n$ we define $\tilde{\mu}_n = \mu_n \circ T_n$.
One can then define a topology using the $L^p$ distance between $\tilde{\mu}_n$ and $\nu$.
The challenge is to define $T_n$ optimally in the sense that as little mass as possible is moved.
In this section we given an overview of the framework proposed in~\cite{trillos15} that describes the topology we use in the sequel.
We start by defining the $p$-OT distance.

\begin{mydef}
\label{def:Prelim:Trans:Metric}
If $1\leq p<\infty$ then the $p$-OT distance between $P,Q\in\mathcal{P}(X)$ is defined by
\begin{equation} \label{eq:Prelim:Trans:dp}
d_p(P,Q) = \min\left\{ \left(\int_{X^2} |x-y|^p \; \pi(\mathrm{d}x,\mathrm{d}y) \right)^{\frac{1}{p}} : \pi \in \Gamma(P,Q) \right\}
\end{equation}
where $\Gamma(P,Q)$ is the set of couplings between $P$ and $Q$, i.e. the set of probability measures on $X\times X$ such that the first marginal is $P$ and the second marginal is $Q$.

If $p=\infty$ then the $\infty$-OT distance between $P,Q\in\mathcal{P}(X)$ is defined by
\begin{equation} \label{eq:Prelim:Trans:dinfty}
d_\infty(P,Q) = \min \left\{ \esssup_\pi\left\{|x-y|: (x,y)\in X\times X\right\} : \pi\in \Gamma(P,Q) \right\}.
\end{equation}
\end{mydef}

The minimization problem in~\eqref{eq:Prelim:Trans:dp} and~\eqref{eq:Prelim:Trans:dinfty} is known as Kantorovich's optimal transportation problem.
The minimization is convex and and therefore the minimum is achieved \cite{champion08,villani03}.
One can also show that $d_p$ defines a metric on the space of probability measures.
Elements $\pi\in\Gamma(P,Q)$ are called transference plans.
The distance $d_2$ is also known as the Wasserstein metric and $d_\infty$ the $\infty$-transportation distance.
For bounded $X\subset\mathbb{R}^d$ convergence in $d_p$ (for $1\leq p<\infty$) is equivalent to the weak convergence of probability measures~\cite{villani03} and therefore with probability one, $d_p(P_n,P)\to 0$ where $P_n$ is the empirical measure.

When $P$ has density with respect to the Lebesgue measure then the Kantorovich minimization problem is equivalent to the Monge optimal transportation problem ~\cite{gangbo96}:
\[ \text{Minimize } \int_X |x-T(x)|^p \; P(\mathrm{d}x) \quad \text{over all measurable maps } T \text{ such that } T_\# P = Q \]
where the push forward measure is defined by
\[ T_{\#}P(A) = P(T^{-1}(A)) \]
for any $A\in\mathcal{B}(X)$.
If $Q=T_{\#} P$ then we call $T$ a transportation map between $P$ and $Q$.

Let $P_n$ be the empirical measure and $1\leq p < \infty$ then from $d_p(P_n,P)\to 0$ (almost surely) one can immediately infer the existence of a sequence of transport plans such that
\begin{equation} \label{eq:Prelim:Trans:Stag}
\| \mathrm{Id} - T_n \|_{L^p(X;P)}^p = \int_X |x-T_n(x)|^p \; P(\mathrm{d}x) \to 0
\end{equation}
as $n\to \infty$.
We call any sequence of transportation maps $\{T_n\}$ that satisfy~\eqref{eq:Prelim:Trans:Stag} stagnating.

In the next definition we use stagnating transport maps to define piecewise constant approximations of functions on $\Psi_n$ in order to define a suitable notion of convergence.

\begin{mydef}
\label{def:Prelim:Trans:TLpConv}
Let $\mu_n\in L^p(\Psi_n)=L^p(X;P_n)$ and $\mu\in L^p(X;P)$ where $P$ is a probability measure and $P_n=\frac{1}{n} \sum_{i=1}^n \delta_{\xi_i}$ is the empirical measure.
We say $\mu_n\to \mu$ in $TL^p(X)$ if
\begin{equation} \label{eq:Prelim:Trans:ConvTL1}
\| \mu_n\circ T_n - \mu \|_{L^p(X;P)}^p = \int_X |\mu_n(T_n(x)) - \mu(x) |^p \; P(\textrm{d}x) \to 0
\end{equation}
for any sequence of stagnating transportation maps $T_n:X \to \Psi_n$.
Similarly $\mu_n$ is bounded in $TL^p$ if $\| \mu_n\circ T_n \|_{L^p}$ is bounded and $\mu_n$ is precompact in $TL^p$ if $\mu_n\circ T_n$ is precompact in $L^p$.
\end{mydef}

One can show that if~\eqref{eq:Prelim:Trans:ConvTL1} holds for one sequence of stagnating transport maps then it holds for any sequence of stagnating transport maps~\cite[Lemma 3.5]{trillos15}.

Since it is assumed that $P$ has density $\rho$ which is bounded above and below by positive constants then~\eqref{eq:Prelim:Trans:Stag} is equivalent to $\|\mathrm{Id} - T_n\|_{L^p(X)} \to 0$ and $L^p(X;P)=L^p(X)$.
This paper focus on the case where $p=1$ however it is straightforward to consider the case $1\leq p < \infty$, see Section~\ref{subsec:MainRes:Comments}.

Now consider an arbitrary $T:X\to X$ and a measurable $\varphi:X\to [0,\infty)$.
Recall that
\[ \int_X \varphi(x) \; T_{\#}P(\mathrm{d} x) := \sup \left\{ \int_X s(x) T_{\#}P(\mathrm{d} x) : 0\leq s\leq \varphi \text{ and } s \text{ is simple} \right\}. \]
If $s(x) = \sum_{i=1}^N a_i \delta_{U_i}(x)$ where $a_i = s(x)$ for any $x\in U_i$ then
\[ \int_X s(x) T_{\#}P(\mathrm{d} x) = \sum_{i=1}^N a_i T_{\#} P(U_i) = \sum_{i=1}^N a_i P(V_i) \]
for $V_i=T^{-1}(U_i)$.
Note that $a_i = s(x)$ for any $x\in T(U_i)$.
From this it is not hard to see the following change of variables formula:
\begin{equation} \label{eq:Prelim:Trans:ChangeVar}
\int_X \varphi(x) \; T_{\#} P(\mathrm{d} x) = \int_X \varphi(T(x)) \; P (\mathrm{d} x).
\end{equation}
A particularly useful version of this will be when $T_{\#} P(\mathrm{d} x) = P_n(\mathrm{d} x)$ where $P_n$ is the empirical measure.
In which case~\eqref{eq:Prelim:Trans:ChangeVar} implies
\[ \frac{1}{n} \sum_{i=1}^n \varphi(\xi_i) = \int_X \varphi(T(x)) \; P(\mathrm{d} x). \]

As an aside one can generalize the $TL^p$ framework to pairs $(\mu,P)$ where $\mu\in L^p(X;P)$. 
Let us define
\[ d_{TL^p}((P,\mu),(Q,\zeta)) = \inf_{\pi\in\Gamma(P,Q)} \left\{ \left(\int_{X\times Y} |x-y|^p \; \pi(\mathrm{d}x,\mathrm{d}y)\right)^{\frac{1}{p}} + \left( \int_{X \times Y} |\mu(x)-\zeta(y)|^p \; \pi(\mathrm{d}x,\mathrm{d}y) \right)^{\frac{1}{p}} \right\}. \]
Let $P$ have density with respect to the Lebesgue measure and take a sequence of measures $P_n$ defined on a common space $X$ (where we do not assume that $P_n$ is the empirical measure).
Then $(P_n,f_n)\to (P,f)$ in $TL^p$ is equivalent to weak convergence of measures (due to the first term) and $\| \mu_n\circ T_n - \mu \|_{L^p(X;P)}\to 0$ (due to the second term), see~\cite[Proposition 3.6]{trillos15}.
Since we are working with the empirical measure then with probability one $P_n$ converges weakly to $P$.
Hence the first term plays no role in this paper and so is not included.

Our proofs require a bound on the supremum norm of $T_n-\mathrm{Id}$ given by the following theorem.

\begin{theorem}
\label{thm:Prelim:Trans:TransMapBound}
Let $X\subset\mathbb{R}^d$ with $d\geq 1$ be open, connected and bounded with Lipschitz boundary.
Let $P$ be a probability measure on $X$ with density (with respect to Lebesgue) $\rho$ which is bounded above and below by positive constants.
Let $\xi_1,\xi_2,\dots$ be a sequence of independent random variables with distribution $P$ and let $P_n$ be the empirical measure.
Then there exists a constant $C>0$ such that almost surely there exists a sequence of transportation maps $\{T_n\}_{n=1}^\infty$ from $P$ to $P_n$ such that
\[ \limsup_{n\to \infty} \frac{\|T_n-\mathrm{Id}\|_{L^\infty(X)}}{\delta_n} \leq C \]
where
\[ \delta_n = \left\{ \begin{array}{ll} \frac{\sqrt{\log\log n}}{\sqrt{n}} & \text{if } d=1 \\ \frac{(\log n)^{\frac{3}{4}}}{\sqrt{n}} & \text{if } d=2 \\ \frac{(\log n)^{\frac{1}{d}}}{n^{\frac{1}{d}}} & \text{if } d\geq 3. \end{array} \right. \]
\end{theorem}

\begin{proof}
The proof for $d=2$ and $d\geq 3$ can be found in~\cite{trillos15b}.
For $d=1$ the result follows from considering the transportation map defined by $T_n(x)=x^{(i)}_n$ for $x\in (x^{(i-1)}_n,x^{(i)}_n]$ where $x^{(i)}_n$ is defined by $P((-\infty,x_n^{(i)}]) = \frac{i}{n}$ for $i=1,\dots,n-1$ and $x^{(0)}_n=-\infty$ and $x_n^{(n)}=\infty$.
One then has $\|T_n-\mathrm{Id}\|_{L^\infty(X)} = \|P_n-P\|_{L^\infty(X)}$ which by the law of the iterated logarithm has the stated rate of convergence.
\end{proof}

\section{The Compactness Property \label{sec:Compact}}

In this section we prove the first part of Theorem~\ref{thm:MainRes:Cons} by establishing that sequences bounded in $\mathcal{E}_n$ are precompact in $TL^1$ with cluster points in $L^1(X;\{0,1\})$.
Our proofs compare $\mathcal{E}_n$ to its continuous analogue $\mathcal{C}_\epsilon:L^1(X) \to [0,\infty]$ defined by
\begin{equation} \label{eq:Compact:Ec}
\mathcal{C}_\epsilon(\mu) = \frac{1}{\epsilon} \int_X V(\mu(x)) \rho(x) \; \mathrm{d} x + \frac{1}{\epsilon^{d+1}} \int_{X^2} \varphi\circ\pi\left(\frac{x-y}{\epsilon}\right) \left|\mu(x) - \mu(y)\right| \rho(x) \rho(y) \; \mathrm{d}x \; \mathrm{d}y.
\end{equation}
The transport map $T_n$ between the measures $P_n$ and $P$ is used to compare a function $\mu_n:\Psi_n\to \mathbb{R}$ to its continuous version $\tilde{\mu}_n:X\to \mathbb{R}$, i.e. $\tilde{\mu}_n = \mu_n \circ T_n$.
One then uses standard results to conclude the compactness of $\tilde{\mu}_n$ in $L^1$ and show that this implies compactness of $\mu_n$ in $TL^1$.

\begin{proposition}
\label{prop:Compact:ComSecondTermData}
Under the same conditions as Theorem~\ref{thm:MainRes:Cons}.
If $\mu_n\in L^1(\Psi_n)$ is a sequence with
\[ \sup_{n\in\mathbb{N}} \mathcal{E}_n(\mu_n)<\infty \]
then, with probability one, there exists a subsequence $\mu_{n_m}$ and $\mu\in L^1(X;\{0,1\})$ such that $\mu_{n_m}\to \mu$ in $TL^1$.
\end{proposition}

\begin{proof}
Recall the following preliminary compactness result.
If $\{\tilde{\mu}_n\}_{n=1}^\infty$ is a sequence in $L^1(X)$ such that
\begin{equation} \label{eq:Compact:cont-fe-seq}
\sup_{n\in\mathbb{N}} \mathcal{C}_{\epsilon_n}(\tilde{\mu}_n) < \infty,
\end{equation}
where $\mathcal{C}_{\epsilon_n}:L^1(X)\to [0,\infty]$ is defined by~\eqref{eq:Compact:Ec},
then there exists a subsequence $\tilde{\mu}_{n_m}$ and $\mu\in L^1(X;\{0,1\})$ such that $\tilde{\mu}_{n_m}\to \mu$ in $L^1$.
A proof can be found, for example, in \cite{alberti98}.

For clarity we will denote the dependence of $\eta=\varphi\circ\pi$ on $\mathcal{E}_n$ by $\mathcal{E}_n(\cdot;\eta)$ and let $\eta_\epsilon(x) = \frac{1}{\epsilon^d}\eta(x/\epsilon)$.
Since $\eta$ is continuous at 0 and $\eta(0)>0$ there exists $b>0$ and $a>0$ such that $\eta(x)\geq a$ for all $|x|<b$.
Define $\tilde{\eta}$ by $\tilde{\eta}(x) = a$ for $|x|<b$ and $\tilde{\eta}(x) = 0$ otherwise.
As $\tilde{\eta}\leq \eta$ then $\mathcal{E}_n(\mu_n;\eta)\geq \mathcal{E}_n(\mu_n;\tilde{\eta})$.

Let $T_n$ be such that $T_{n\#}P = P_n$ and the conclusions of Theorem~\ref{thm:Prelim:Trans:TransMapBound} hold.
We want to show $\{\mu_n\circ T_n\}_{n=1}^\infty$ satisfies
\begin{equation} \label{eq:Compact:munTnEnbound}
\sup_{n\in\mathbb{N}} \mathcal{C}_{\tilde{\epsilon}_n}(\mu_n\circ T_n ;\tilde{\eta}) < \infty
\end{equation}
for a sequence $\tilde{\epsilon}_n>0$ with $\tilde{\epsilon}_n\to 0$ and $\frac{\epsilon_n}{\tilde{\epsilon}_n}\to 1$ that will be chosen shortly.
If so then by~\eqref{eq:Compact:cont-fe-seq} there exists a subsequence $\mu_{n_m}\circ T_{n_m}$ and $\mu\in L^1(X;\{0,1\})$ such that $\mu_{n_m}\circ T_{n_m} \to \mu$ in $L^1$ and therefore $\mu_{n_m}\to \mu$ in $TL^1$.
To show~\eqref{eq:Compact:munTnEnbound} we write
\begin{align*}
\mathcal{C}_{\tilde{\epsilon}_n}(\mu_n\circ T_n;\tilde{\eta}) & = \frac{1}{\tilde{\epsilon}_n}\int_X V(\mu_n(T_n(x)) \rho(x) \; \mathrm{d} x \\
 & \quad \quad \quad + \frac{1}{\tilde{\epsilon}_n^{d+1}} \int_{X^2} \tilde{\eta}\left(\frac{y-x}{\tilde{\epsilon}_n}\right) \left| \mu_n(T_n(x)) - \mu_n(T_n(y)) \right| \rho(x) \rho(y) \; \mathrm{d} y \; \mathrm{d} x.
\end{align*}
The first term is uniformly bounded, since by~\eqref{eq:Prelim:Trans:ChangeVar}
\[ \frac{1}{\tilde{\epsilon}_n}\int_X V(\mu_n(T_n(x)) \rho(x) \; \mathrm{d} x = \frac{\epsilon_n}{\tilde{\epsilon}_n} \frac{1}{\epsilon_n} \sum_{i=1}^n V(\mu_n(\xi_i)). \]
Assume that $\left|\frac{y-x}{\tilde{\epsilon}_n}\right| < b$ then
\begin{align*}
\left| T_n(x) - T_n(y)\right| & \leq \left| T_n(x) - x \right| + \left| x-y \right| + \left| y - T_n(y) \right| \\
 & \leq 2\|\mathrm{Id} - T_n\|_{L^\infty(X)} + \left|x-y\right| \\
 & \leq 2\|\mathrm{Id} - T_n\|_{L^\infty(X)} + b\tilde{\epsilon}_n.
\end{align*}
Choose $\tilde{\epsilon}_n$ satisfying
\[ 2\|T_n-\mathrm{Id}\|_{L^\infty(X)} + b\tilde{\epsilon}_n = b\epsilon_n, \]
i.e. $\tilde{\epsilon}_n = \epsilon_n - \frac{2\|T_n-\mathrm{Id}\|_{L^\infty(X)}}{b}$.
By the decay assumption on $\epsilon_n$ for $n$ sufficiently large (with probability one) $\tilde{\epsilon}_n>0$, $\tilde{\epsilon}_n\to 0$ and $\frac{\tilde{\epsilon}_n}{\epsilon_n}\to 1$.
Also
\[ \tilde{\eta}\left( \frac{x-y}{\tilde{\epsilon}_n} \right) = a \Rightarrow \tilde{\eta}\left( \frac{T_n(x)-T_n(y)}{\epsilon_n} \right) = a. \]
Therefore
\[ \tilde{\eta}_{\tilde{\epsilon}_n}(y-x) = \frac{1}{\tilde{\epsilon}_n^d} \tilde{\eta}\left( \frac{x-y}{\tilde{\epsilon}_n} \right) \leq \frac{1}{\tilde{\epsilon}_n^d} \tilde{\eta}\left( \frac{T_n(x)-T_n(y)}{\epsilon_n} \right) = \frac{\epsilon_n^d}{\tilde{\epsilon}_n^d} \tilde{\eta}_{\epsilon_n}\left(T_n(x) - T_n(y)\right). \]
So,
\begin{align*}
& \frac{1}{\tilde{\epsilon}_n} \int_{X^2} \tilde{\eta}_{\tilde{\epsilon}_n}(y-x) \left| \mu_n(T_n(x)) - \mu_n(T_n(y)) \right| \rho(x) \rho(y) \; \mathrm{d} y \; \mathrm{d} x \\
& \quad \quad \quad \leq \frac{\epsilon_n^d}{\tilde{\epsilon}_n^{d+1}} \int_{X^2} \tilde{\eta}_{\epsilon_n}\left(T_n(x) - T_n(y)\right) \left| \mu_n(T_n(x)) - \mu_n(T_n(y)) \right| \rho(x) \rho(y) \; \mathrm{d} y \; \mathrm{d} x \\
& \quad \quad \quad = \frac{\epsilon_n^d}{\tilde{\epsilon}_n^{d+1}} \sum_{i,j} \tilde{\eta}_{\epsilon_n}\left(\xi_i - \xi_j\right) \left| \mu_n(\xi_i) - \mu_n(\xi_j) \right| \quad \text{by~\eqref{eq:Prelim:Trans:ChangeVar}} \\
& \quad \quad \quad = \frac{\epsilon_n^{d+1}}{\tilde{\epsilon}_n^{d+1}} \mathcal{E}_n(\mu_n;\tilde{\eta}) \\
& \quad \quad \quad \leq \frac{\epsilon_n^{d+1}}{\tilde{\epsilon}_n^{d+1}} \mathcal{E}_n(\mu_n;\eta).
\end{align*}
It follows that the second term is also uniformly bounded in $n$.
\end{proof}

\section{\texorpdfstring{$\Gamma$}{Gamma}-Convergence \label{sec:Gamma}}

The main result of this section is Theorem~\ref{thm:Gamma:EntoEinfty}, which states that $\mathcal{E}_n$ $\Gamma$-converge to $\mathcal{E}_\infty$.

\begin{theorem}
\label{thm:Gamma:EntoEinfty}
Under the same conditions as Theorem~\ref{thm:MainRes:Cons}
\[ \mathcal{E}_\infty = \Glim_{n\to \infty} \mathcal{E}_n \]
in the $TL^1$ sense and with probability one.
\end{theorem}

The proof is a consequence of Lemmas~\ref{lem:Gamma:EntoEinftyliminf} and~\ref{lem:Gamma:EntoEinftylimsup}.

\begin{lemma}[The lim inf inequality]
\label{lem:Gamma:EntoEinftyliminf}
Under the same conditions as Theorem~\ref{thm:MainRes:Cons} if $\mu\in L^1(X)$ and $\mu_n\to \mu$ in $TL^1$ then
\[ \mathcal{E}_\infty(\mu) \leq \liminf_{n\to \infty} \mathcal{E}_n(\mu_n) \]
with probability one.
\end{lemma}

\begin{proof}
Recall $\eta=\varphi\circ\pi$.
Let $\mu_n\in L^1(\Psi_n), \mu\in L^1(X)$ with $\mu_n\to \mu$ in $TL^1$.
Let $\nu_n = \mu_n \circ T_n \in L^1(X)$ where $T_n:X\to \Psi_n$ is as in Theorem~\ref{thm:Prelim:Trans:TransMapBound} (with probability one) so $\nu_n\to \mu$ in $L^1(X)$.
Without loss of generality we assume that
\[ \liminf_{n\to \infty} \mathcal{E}_n(\mu_n) < \infty \]
else there is nothing to prove.
By Theorem~\ref{prop:Compact:ComSecondTermData} $\mu\in L^1(X;\{0,1\})$ hence the proof is complete if
\begin{equation} \label{eq:Gamma:GTVliminfTV}
\liminf_{n\to \infty} GTV_n(\mu_n) \geq TV(\mu;\rho,\eta)
\end{equation}
where $GTV_n$ is defined by~\eqref{eq:Intro:Finite:TVG}.

We show~\eqref{eq:Gamma:GTVliminfTV} in two steps.
\begin{description}
\item[Step 1.] Assume $\rho$ is Lipschitz continuous on $X$.
\item[Step 2.] Generalise to continuous densities.
\end{description}

\paragraph*{Step 1.}
Let $X^\prime$ be a compact subset of $X$.
We have
\begin{align*}
GTV_n(\mu_n) & = \frac{1}{\epsilon_n} \frac{1}{n^2} \sum_{i,j} W_{ij} \left| \mu_n(\xi_i) - \mu_n(\xi_j) \right| \\
 & = \frac{1}{\epsilon_n^{d+1}} \int_{X^2} \eta\left(\frac{T_n(x)-T_n(y)}{\epsilon_n}\right) \left| \nu_n(x)-\nu_n(y) \right| \rho(x) \rho(y) \; \mathrm{d} x \; \mathrm{d} y \quad \text{using~\eqref{eq:Prelim:Trans:ChangeVar}} \\
 & \geq \frac{1}{\epsilon_n^{d+1}} \int_{X^\prime} \int_X \eta\left(\frac{T_n(x)-T_n(y)}{\epsilon_n}\right) \left| \nu_n(x)-\nu_n(y) \right| \rho(x) \rho(y) \; \mathrm{d} x \; \mathrm{d} y \\
 & = \frac{1}{\epsilon_n} \int_{X^\prime} \int_{y+\epsilon_n z\in X} \eta\left(\frac{T_n(y+\epsilon_n z)-T_n(y)}{\epsilon_n}\right) \left| \nu_n(y+\epsilon_n z)-\nu_n(y) \right| \rho(y+\epsilon_n z) \rho(y) \; \mathrm{d} z \; \mathrm{d} y \\
 & = \frac{1}{\epsilon_n} \int_{X^\prime} \int_{y+\epsilon_n z\in X} \eta\left(\frac{T_n(y+\epsilon_n z) - T_n(y)}{\epsilon_n}\right) \left| \nu_n(y+\epsilon_n z)-\nu_n(y) \right| \rho^2(y) \; \mathrm{d} z \; \mathrm{d}y + a_n
\end{align*}
where
\[ a_n = \frac{1}{\epsilon_n} \int_{X^\prime} \int_{y+\epsilon_n z\in X} \eta\left(\frac{T_n(y+\epsilon_n z) - T_n(y)}{\epsilon_n}\right) \left| \nu_n(y+\epsilon_n z)-\nu_n(y) \right| \rho(y) \left( \rho(y+\epsilon_n z) - \rho(y) \right) \; \mathrm{d} z \; \mathrm{d}y. \]
Assume that the support of $\varphi$ is contained within $B(0,M)$ then the following calculation shows that for any $y\in X$ that if $|z|\geq \tilde{M}$ where $\tilde{M}=\frac{M}{\alpha} + 2\sup_{n\in\mathbb{N}}\frac{\|T_n-\mathrm{Id}\|_{L^\infty}}{\epsilon_n} <\infty$ then $\eta\left(\frac{T_n(y+\epsilon_n z) - T_n(y)}{\epsilon_n}\right) = 0$:
\begin{align*}
M & \leq \alpha \epsilon_n |z| - 2 \alpha \frac{\|T_n-\mathrm{Id}\|_{L^\infty}}{\epsilon_n} \\
 & \leq \frac{\alpha}{\epsilon_n} \bigg( \epsilon_n |z| - |T_n(y+\epsilon_n z)-T_n(y) - \epsilon_n z | \bigg) \\
 & \leq \frac{\alpha}{\epsilon_n} \bigg| T_n(y+\epsilon_n z) - T_n(y) \bigg| \\
 & \leq \pi\left(\frac{T_n(y+\epsilon_n z)-T_n(y)}{\epsilon_n}\right).
\end{align*}

It follows that
\begin{align*}
|a_n| & \leq \frac{\mathrm{Lip}(\rho) \tilde{M}}{\inf_{x\in X} \rho(x)} \int_X \int_{y+\epsilon z \in X} \eta\left(\frac{T_n(y+\epsilon_n z) - T_n(y)}{\epsilon_n}\right) \left| \nu_n(y+\epsilon_n z)-\nu_n(y) \right| \rho(y) \rho(y+\epsilon_n z) \; \mathrm{d} z \; \mathrm{d} y \\
 & = \frac{\tilde{M}\mathrm{Lip}(\rho)}{\inf_{x\in X} \rho(x)} \epsilon_n GTV_n(\mu_n) \\
 & \to 0 \quad \text{as } n \to \infty.
\end{align*}

Let $\alpha_n$ and $c_n$ be as in Definition~\ref{def:MainRes:Admissible} for $\delta=\frac{2\|T_n-\mathrm{Id}\|_{L^\infty}}{\epsilon_n}$ then since
\[ \left| \frac{T_n(y+\epsilon_n z) - T_n(y)}{\epsilon_n} - z \right| \leq \frac{1}{\epsilon_n} \left( \left| T_n(y+\epsilon_n z) - y-\epsilon_n z \right| + \left| T_n(y) - y \right| \right) \leq \frac{2\|T_n-\mathrm{Id}\|_{L^\infty}}{\epsilon_n} \]
we have that
\[ \eta\left(\frac{T_n(y+\epsilon_n z)-T_n(y)}{\epsilon_n}\right) \geq c_n \eta\left(\alpha_n z\right). \]
So,
\begin{align*}
GTV_n(\mu_n) & \geq \frac{c_n}{\epsilon_n} \int_{X^\prime} \int_{y+\epsilon_n x \in X} \eta(\alpha_n z) \left| \nu_n(y+\epsilon_n z) - \nu_n(y)\right| \rho^2(y) \; \mathrm{d} z \; \mathrm{d} y + o(1) \\
 & = \frac{c_n}{\alpha_n^d \epsilon_n} \int_{X^\prime} \int_{y+\frac{\epsilon_n \tilde{z}}{\alpha_n}\in X} \eta(\tilde{z}) \left| \nu_n\left(y+\frac{\epsilon_n\tilde{z}}{\alpha_n}\right) - \nu_n(y) \right| \rho^2(y) \; \mathrm{d} \tilde{z} \; \mathrm{d} y + o(1) \\
 & = \frac{c_n}{\alpha^d_n \epsilon_n} \int_{X^\prime} \int_{\mathbb{R}^d} \eta(\tilde{z}) \left| \nu_n\left(y+\frac{\epsilon_n\tilde{z}}{\alpha_n}\right) - \nu_n(y) \right| \rho^2(y) \; \mathrm{d} \tilde{z} \; \mathrm{d} y - b_n + o(1)
\end{align*}
where
\[ b_n = \frac{c_n}{\alpha^d_n \epsilon_n} \int_{X^\prime} \int_{y+\frac{\epsilon_n\tilde{z}}{\alpha_n}\not\in X} \eta(\tilde{z}) \left| \nu_n\left(y+\frac{\epsilon_n\tilde{z}}{\alpha_n}\right) - \nu_n(y) \right| \rho^2(y) \; \mathrm{d} \tilde{z} \; \mathrm{d} y \]
By Fatou's lemma, Proposition~\ref{prop:Prelim:TV:Liminf} and since $c_n,\alpha_n\to 1$ we have
\begin{align*}
& \liminf_{n\to \infty} \frac{c_n}{\alpha_n^d\epsilon_n} \int_{X^\prime} \int_{\mathbb{R}^d} \eta(\tilde{z}) \left| \nu_n\left(y+\frac{\epsilon_n \tilde{z}}{\alpha_n}\right)-\nu_n(y) \right| \rho^2(y) \; \mathrm{d} \tilde{z} \; \mathrm{d}y \\
& \quad \quad \quad \geq \int_{X^\prime} \eta(z) TV_z(\mu;\rho) \; \mathrm{d} z \\
& \quad \quad \quad = TV(\mu;\rho,\eta,X^\prime).
\end{align*}

Since $X^\prime$ and $X^c$ are both closed and disjoint then $\tau:=\mathrm{dist}(X^\prime,X^c)>0$ and therefore if $y\in X^\prime$ and $y+\frac{\epsilon_n \tilde{z}}{\alpha_n}\not\in X$ then $|\tilde{z}|\geq \frac{\alpha_n\tau}{\epsilon_n}$.
Choose $n$ sufficiently large such that $\frac{\alpha_n\tau}{\epsilon_n}\geq M$ where the support of $\eta$ is contained in $B(0,M)$.
Then for any pair $y,\tilde{z}$ satisfying the above condition we have $\eta(z)=\varphi\circ\pi(\tilde{z})=0$.
Hence $b_n=0$ for $n$ sufficiently large.

We have shown that
\[ \liminf_{n\to \infty} GTV_n(\mu_n) \geq TV(\mu;\rho,\varphi\circ\pi,X^\prime). \]
If we consider a sequence $X^\prime_m$ such that $X^\prime_m\subset X^\prime_{m+1}$ and $\mathbb{I}_{X\prime}\to \mathbb{I}_X$ pointwise then by the monotone convergence theorem $TV(\mu;\rho,\varphi\circ\pi,X^\prime_m) \to TV(\mu;\rho,\varphi\circ\pi,X)$.
This completes step 1.

\paragraph*{Step 2.}
Denote the dependence of $\rho$ on $GTV_n$ by $GTV_n(\cdot;\rho)$.
Assume $\rho:X\to [0,\infty)$ is continuous and let $\rho_k:\mathbb{R}^d\to [0,\infty)$ be defined by
\begin{equation} \label{eq:Gamma:rhok}
\rho_k(x) = \left\{ \begin{array}{ll} \inf_{y\in X} \left( \rho(y) + k|x-y| \right) & \text{if } x\in X \\ 0 & \text{otherwise.} \end{array} \right.
\end{equation}
Clearly $\rho_k(x)\leq \rho(x)$ for all $x\in X$.
For any $y\in X\setminus B\left(x,\frac{\rho(x)}{k}\right)$ we have $|x-y|\geq \frac{\rho}{k}$ and therefore
\[ \rho(y) + k|x-y| \geq \rho(y) + \rho(x) > \rho(x) \geq \rho_k(x). \]
Then it follows that any (approximate) minimizer $y\in X$ of~\eqref{eq:Gamma:rhok} must be contained in $B\left(x,\frac{\rho(x)}{k}\right)$.
Hence $\rho_k(x) = \inf\left\{ \rho(y) + k|x-y| \, : \, y \in B\left(x,\frac{\rho(x)}{k}\right) \right\}$ and therefore
\[ \rho(x) \geq \rho_k(x) \geq \inf\left\{ \rho(y): y\in B\left(x,\frac{\rho(x)}{k}\right) \right\}. \]
As $\rho$ is bounded above on $X$ then the previous inequality implies $\rho_k(x)\to \rho(x)$ for each $x\in X$.
It is also clear that $\rho_k(x) \geq \inf_{x\in X} \rho(x) >0$.
Furthermore, for $x,z\in X$
\begin{align*}
\rho_k(x) - \rho_k(z) & = \inf_{y_1\in X} \sup_{y_2\in X} \rho(y_1) - \rho(y_2) + k\left(|x-y_1| - |z-y_2| \right) \\
 & \leq \sup_{y_2\in X} k\left(|x-y_2| - |z-y_2| \right) \\
 & \leq k|x-z|
\end{align*}
so $\rho_k$ is Lipschitz in $X$.
By step 1
\[ \liminf_{n\to \infty} GTV_n(\mu_n;\rho) \geq \liminf_{n\to \infty} GTV_n(\mu_n;\rho_k) \geq TV(\mu;\rho_k,\eta). \]
By Theorem~\ref{thm:Prelim:TV:Measure}
\[ TV(\mu;\rho_k,\eta) = \int_X \rho_k^2(x) \sigma(x) \; \hat{\lambda}(\mathrm{d} x). \]
And therefore by the monotone convergence theorem one has
\[\lim_{k\to \infty} TV(\mu;\rho_k,\eta) = \int_X \rho(x)^2 \sigma(x) \; \hat{\lambda}( \mathrm{d} x) = TV(\mu;\rho,\eta) \]
which completes the proof.
\end{proof}

For $\mu\not\in L^1(X;\{0,1\})$ the recovery sequence is trivial as $\mathcal{E}_\infty(\mu)=\infty$.
For $\mu\in L^1(X;\{0,1\})$ we show that it is enough to prove the existence of a recovery sequence when $\mu$ is a polyhedral function (defined below).
Recall that $\mathcal{H}^k$ is the $k$-dimensional Hausdorff measure.

\begin{mydef}
\label{def:Gamma:PolyFun}
A ($d$-dimensional) polyhedral set in $\mathbb{R}^d$ is an open set $F$ whose boundary is a Lipschitz manifold contained in the union of finitely many affine hyperplanes.
We say $\mu\in BV(X;\{0,1\})$ is a polyhedral function if there exists a polyhedral set $F$ such that $\partial F$ is transversal to $\partial X$ (i.e. $\mathcal{H}^{d-1}(\partial F\cup \partial X) = 0$) and $\mu(x) = 1$ for $x\in X\cap F$, $\mu(x) = 0$ for $x\in X\setminus F$.
\end{mydef}

\begin{lemma}[The existence of a recovery sequence for Theorem~\ref{thm:Gamma:EntoEinfty}]
\label{lem:Gamma:EntoEinftylimsup}
Under the same conditions as Theorem~\ref{thm:MainRes:Cons} for any $\mu\in L^1(X)$ there exists a sequence $\mu_n\to \mu$ in $TL^1$ such that
\begin{equation} \label{eq:Gamma:limsupEntoEinfty}
\mathcal{E}_\infty(\mu) \geq \limsup_{n\to \infty} \mathcal{E}_n(\mu_n)
\end{equation}
with probability one.
\end{lemma}

\begin{proof}
Without loss of generality assume $\mu\in BV(X;\rho,\eta)\cap L^1(X;\{0,1\})$.
By the following (diagonalization) argument it is enough to prove the lemma for polyhedral functions.
Suppose the lemma holds for polyhedral functions and let $\mu\in BV(X;\rho,\eta) \cap L^1(X;\{0,1\})$.
There exists a sequence of polyhedral functions $\mu_m\to \mu$ in $L^1$ and $TV(\mu_m;\rho,\eta)\to TV(\mu;\rho,\eta)$, for example see~\cite[Section 9.4.1 Lemma 1]{maz'ya10}.
By passing to a subsequence we may assume that
\[ \| \mu_m - \mu \|_{L^1} \leq \frac{1}{m} \quad \quad \text{and} \quad \quad \left| \mathcal{E}_\infty(\mu_m) - \mathcal{E}_\infty(\mu) \right| \leq \frac{1}{m}. \]
Now for each $m$ there exists a sequence $\mu^{(m)}_n$ and a limit $\mu^{(m)}$ (as $n\to \infty$ for each $m$) such that
\[ \mu^{(m)}_n \to \mu^{(m)} \quad \text{in } L^1 \text{ as } n\to \infty \quad \quad \text{and} \quad \quad \mathcal{E}_\infty(\mu^{(m)}) \geq \limsup_{n\to \infty} \mathcal{E}_n(\mu^{(m)}_n). \]
For each $m$ there exists $N_m$ such that
\[ \left\| \mu^{(m)}_n - \mu^{(m)} \right\|_{L^1} \leq \frac{1}{m} \quad \quad \text{and} \quad \quad \mathcal{E}_\infty(\mu^{(m)}) \geq \mathcal{E}_n(\mu^{(m)}_n) - \frac{1}{m} \]
for all $n \geq N_m$.
Let $\mu_m=\mu^{(m)}_{N_m}$ then we have
\[ \left\| \mu_m - \mu \right\|_{L^1} \leq \left\| \mu^{(m)}_{N_m} - \mu^{(m)} \right\|_{L^1} + \left\| \mu^{(m)} - \mu \right\|_{L^1} \leq \frac{2}{m} \quad \quad \text{and} \quad \quad \mathcal{E}_\infty(\mu) \geq \mathcal{E}_\infty(\mu^{(m)}) - \frac{1}{m} \geq \mathcal{E}_n(\mu_m) - \frac{2}{m}. \]
Hence it is enough to prove the theorem for polyhedral functions.


Therefore assume $\mu\in BV(X;\{0,1\})$ is a polyhedral function corresponding to the polyhedral set $F$, i.e. $\mu = \mathbb{I}_F$.
Let $\mu_n$ be the restriction of $\mu$ to $\Psi_n$.
Define $T_n$ as in Theorem~\ref{thm:Prelim:Trans:TransMapBound} and use it to create a partition of $X$, $X=\cup_{i=1}^n T_n^{-1}(\xi_i)$.
If $x,y\in T_n^{-1}(\xi_i)$ then
\[ | x-y | \leq |x-T_n(x)| + |T_n(x) - T_n(y)| + |T_n(y) - y| \leq 2\|\mathrm{Id} - T_n\|_{L^\infty(X)}. \]
Let $x\in T_n^{-1}(\xi_i)$ and assume $\mathrm{dist}(\partial F,x)> 2\|\mathrm{Id} - T_n\|_{L^\infty}$.
If $y\in T_n^{-1}(\xi_i)$ then $\mu(y)=\mu(x)$ (since $T_n^{-1}(\xi_i)\subset B(x,2\|\mathrm{Id}-T_n\|_{L^\infty(X)})$ and $B(x,2\|\mathrm{Id}-T_n\|_{L^\infty(X)})\cap \partial F = \emptyset$).
Therefore
\[ \int_{T_n^{-1}(\xi_i)} \left| \mu_n(T_n(y)) - \mu(y) \right| \rho(y) \; \mathrm{d} y = 0. \]
In particular
\[ \int_X \left| \mu_n(T_n(y)) - \mu(y) \right| \rho(y) \; \mathrm{d} y = \int_{X_n} \left| \mu_n(T_n(y)) - \mu(y) \right| \rho(y) \; \mathrm{d} y \leq \|\rho\|_{L^\infty(X)} \mathrm{Vol}(X_n) \]
where
\[ X_n = \left\{ y\in X \; : \; \mathrm{dist}(\partial F,y) \leq 2\|\mathrm{Id} - T_n\|_{L^\infty(X)} \right\}. \]
Clearly $\mathrm{Vol}(X_n) = O(\|\mathrm{Id} - T_n\|_{L^\infty(X)}) = o(1)$ and therefore $\mu_n\to \mu$ in $TL^1$.
Define $\nu_n=\mu_n\circ T_n$ then since $\nu_n,\mu\in L^1(X;\{0,1\})$ we have that~\eqref{eq:Gamma:limsupEntoEinfty} is equivalent to
\[ TV(\mu;\rho,\eta) \geq \limsup_{n\to \infty} GTV_n(\mu_n). \]

We complete the proof in two steps.
\begin{description}
\item[Step 1.] Assume $\rho$ is Lipschitz on $X$.
\item[Step 2.] Generalise to continuous densities.
\end{description}

\paragraph*{Step 1.}
We can write assumption 3 in Definition~\ref{def:MainRes:Admissible} as for any $\delta>0$ there exists $c_\delta,\alpha_\delta$ such that if $|x-y|<\delta$ then $\eta(x)\leq \frac{1}{c_\delta} \eta\left(\frac{w}{\alpha_\delta}\right)$ and $c_\delta,\alpha_\delta\to 1$ as $\delta\to 0$.

Let $c_n,\alpha_n$ be such constants for $\delta=\frac{2\|T_n-\mathrm{Id}\|_{L^\infty}}{\epsilon_n}$ since
\[ \left| \frac{T_n(x)-T_n(y)}{\epsilon_n} - \frac{x-y}{\epsilon_n} \right| \leq 2\frac{2\|T_n-\mathrm{Id}\|_{L^\infty}}{\epsilon_n} = \delta \]
then
\[ \eta\left(\frac{T_n(x)-T_n(y)}{\epsilon_n}\right) \leq \frac{1}{c_n} \eta\left(\frac{x-y}{\tilde{\epsilon}_n}\right) \]
where $\tilde{\epsilon}=\epsilon_n\alpha_n$.

So,
\begin{align*}
GTV_n(\mu_n) & = \frac{1}{\epsilon_n^{d+1}} \int_{X^2} \eta\left(\frac{T_n(x)-T_n(y)}{\epsilon}\right) \left| \nu_n(x) - \nu_n(y) \right| \rho(x) \rho(y) \; \mathrm{d} x \; \mathrm{d} y \\
 & \leq \frac{\alpha_n^{d+1}}{c_n\tilde{\epsilon}_n^{d+1}} \int_{X^2} \eta\left(\frac{x-y}{\tilde{\epsilon}_n}\right) \left| \nu_n(x) - \nu_n(y) \right| \rho(x) \rho(y) \; \mathrm{d} x \; \mathrm{d} y \\
 & = \frac{\alpha_n^{d+1}}{c_n} CTV_n(\nu_n)
\end{align*}
with $CTV_n$ defined below.
Let us approximate $\mu$ by a sequence $\zeta_n\in C^\infty(\mathbb{R}^d)\cap BV(X)$ such that $\zeta_n\to \mu$ in $L^1(X)$ and $TV(\zeta_n;\rho,\eta) \to TV(\mu;\rho,\eta)$.
Without loss of generality assume that $\zeta_n(x) = 0$ for all $x\in \mathbb{R}^d\setminus X$ and $\|\mu-\zeta_n\|_{L^1(X)} = o(\tilde{\epsilon}_n)$ (by recourse to a subsequence of $\zeta_n$ and relabelling).
Then
\begin{align*}
CTV_n(\zeta_n) & := \frac{1}{\tilde{\epsilon}_n^{d+1}} \int_{X^2} \eta\left(\frac{x-y}{\tilde{\epsilon}_n}\right) \left| \zeta_n(x) - \zeta_n(y) \right| \rho(x) \rho(y) \; \mathrm{d} x \; \mathrm{d} y \\
 & = \frac{1}{\tilde{\epsilon}_n^{d+1}} \int_{X^2} \eta\left(\frac{x-y}{\tilde{\epsilon}_n}\right) \left| \int_0^1 \nabla\zeta_n(y+s(x-y)) \cdot(x-y) \; \mathrm{d} s \right| \rho(x) \rho(y) \; \mathrm{d} x \; \mathrm{d} y \\
 & \leq \frac{1}{\tilde{\epsilon}_n^{d+1}} \int_X \int_0^1 \int_X \eta\left(\frac{x-y}{\tilde{\epsilon}_n}\right) \left| \nabla\zeta_n(y+s(x-y)) \cdot(x-y) \right| \rho(x) \rho(y) \; \mathrm{d} x \; \mathrm{d} s \; \mathrm{d} y \\
 & = \int_X \int_0^1 \int_{\mathcal{Z}_{nhs}} \eta(z) \left|\nabla \zeta_n(h) \cdot z \right| \rho(h+(1-s)\tilde{\epsilon}_n z) \rho(h-\tilde{\epsilon}_n s z) \; \mathrm{d} z \; \mathrm{d} s \; \mathrm{d} h \\
 & \leq TV(\zeta_n;\rho,\eta) + c_n
\end{align*}
where
\begin{align*}
\mathcal{Z}_{nhs} & = \left\{ z\in\mathbb{R}^d : h+(1-s)\tilde{\epsilon}_n z \in X \text{ and } h-\tilde{\epsilon}_n s z \in X \right\} \\
c_n & = \int_X \int_0^1 \int_{\mathcal{Z}_{nhs}} \eta(z) \left| \nabla \zeta_n(h) \cdot z\right| \left( \rho(h+(1-s)\tilde{\epsilon}_n z) \rho(h-\tilde{\epsilon}_n s z) - \rho^2(h) \right) \; \mathrm{d} z \; \mathrm{d} s \; \mathrm{d} h.
\end{align*}
If
\begin{align}
& \lim_{n\to \infty} c_n = 0 \label{eq:Gamma:bn} \\
\text{and} \quad & \lim_{n\to \infty} \left| CTV_n(\nu_n) - CTV_n(\zeta_n) \right| = 0 \label{eq:Gamma:VTVmuzeta}
\end{align}
then
\[ \limsup_{n\to \infty} CTV_n(\nu_n) = \limsup_{n\to \infty} CTV_n(\zeta_n) \leq \limsup_{n\to \infty} TV(\zeta_n;\rho,\eta) = TV(\mu;\rho,\eta). \]
We now show~\eqref{eq:Gamma:bn}.
Since the support of $\eta$ is bounded then there exists $M>0$ such that $\mathrm{spt}(\eta)\subset B(0,M)$.
For any $|z|>M$ we have $\eta(z)=0$ and for any $|z|\leq M$ one can show
\begin{align*}
\left| \rho^2(h) - \rho(h+(1-s)\tilde{\epsilon}_n z)\rho(h+\tilde{\epsilon}_n s z)\right| & \leq \|\rho\|_{L^\infty(X)} \mathrm{Lip}(\rho) \left(|(1-s)\tilde{\epsilon}_n z| + |\tilde{\epsilon}_ns z|\right) \\
 & \leq 2\tilde{\epsilon}_n \|\rho\|_{L^\infty(X)} \mathrm{Lip}(\rho) M.
\end{align*}
Then
\begin{align*}
|c_n| & \leq 2\|\rho\|_{L^\infty(X)} \mathrm{Lip}(\rho) \tilde{\epsilon}_n M \int_X \int_{\mathbb{R}^d} \eta(z) \left| \nabla \zeta_n(h) \cdot z \right| \; \mathrm{d} z \; \mathrm{d} h \\
 & \leq \frac{2\|\rho\|_{L^\infty(X)} \mathrm{Lip}(\rho) \tilde{\epsilon}_n M}{\inf_{x\in X} \rho^2(x)} \int_X \int_{\mathbb{R}^d} \eta(z) \left| \nabla \zeta_n(h) \cdot z \right| \rho^2(h) \; \mathrm{d} z \; \mathrm{d} h \\
 & = \frac{2\|\rho\|_{L^\infty(X)} \mathrm{Lip}(\rho) \tilde{\epsilon}_n M}{\inf_{x\in X} \rho^2(x)} TV(\zeta_n;\rho,\eta).
\end{align*}
To complete step 1 we show~\eqref{eq:Gamma:VTVmuzeta}.
This follows by:
\begin{align*}
& \left| CTV_n(\nu_n) - CTV_n(\zeta_n) \right| \\
& \quad \quad \leq \frac{\|\rho\|_{L^\infty(X)}^2}{\tilde{\epsilon}_n^{d+1}} \int_{X^2} \eta\left(\frac{y-x}{\tilde{\epsilon}_n}\right) \left( \left| \nu_n(x)-\zeta_n(x) \right| + \left| \nu_n(y)-\zeta_n(y) \right| \right) \; \mathrm{d} x \; \mathrm{d} y \\
& \quad \quad \leq \frac{2\|\varphi\|_{L^\infty(\mathbb{R}^d)} \mathrm{Vol}(X) \|\rho\|_{L^\infty(X)}^2}{\tilde{\epsilon}_n} \int_X \left| \nu_n(x)-\zeta_n(x) \right| \; \mathrm{d} y \\
& \quad \quad \leq \frac{2\|\varphi\|_{L^\infty(\mathbb{R}^d)} \mathrm{Vol}(X) \|\rho\|_{L^\infty(X)}^2}{\tilde{\epsilon}_n} \left( \|\nu_n - \mu\|_{L^1(X)} + \|\mu-\zeta_n\|_{L^1(X)} \right) \\
& \quad \quad \to 0
\end{align*}
where the last line follows as $\|\mu-\zeta_n\|_{L^1(X)} = o(\tilde{\epsilon}_n)$ and $\|\nu_n - \mu\|_{L^1(X)}=O(\|T_n-\mathrm{Id}\|_{L^\infty(X)}) = o(\epsilon_n)=o(\tilde{\epsilon}_n)$.

\paragraph*{Step 2.}
Let $GTV(\cdot;\rho)$ be the graph total variation defined using $\rho$.
Let $\rho$ be continuous but not necessarily Lipschitz and define $\rho_k:\mathbb{R}^d\to [0,\infty)$ by
\[ \rho_k(x) = \left\{ \begin{array}{ll} \sup_{y\in X} \rho(y) - k|x-y| & \text{if } x\in X \\ 0 & \text{otherwise.} \end{array} \right. \]
Similarly to Lemma~\ref{lem:Gamma:EntoEinftyliminf} step 2 we can check that $\rho_k$ is bounded above and below by positive constants, Lipschitz continuous on $X$ and converges pointwise to $\rho$ from above.
We have
\[ \limsup_{n\to \infty} GTV_n(\mu_n;\rho) \leq \limsup_{n\to \infty} GTV_n(\mu_n;\rho_k) \leq TV(\mu;\rho_k,\eta). \]
By the monotone convergence theorem and Theorem~\ref{thm:Prelim:TV:Measure} we have $\lim_{k\to \infty} TV(\mu;\rho_k,\eta) = TV(\mu;\rho,\eta)$.
\end{proof}

\section{Preliminary Results for the Rate of Convergence \label{sec:Rate}}

There are two main sources of error defined to be $|\mathcal{E}_n(\mu)-\mathcal{E}_\infty(\mu)|$.
The first is due to statistical fluctuations, e.g. data points that lie in the tails that have a large impact when $n$ is small.
The second is due to systematic bias.
Systematic bias is due to several factors such as details of the scaling and the geometry of the problem.
In this section, by taking the expectation over the data, we only consider the second source of error.
In this section we will restrict ourselves to when $\mu=\mathbb{I}_E$ is a polyhedral function (see Definition~\ref{def:Gamma:PolyFun}).

We first fix our notation.
The boundary of a polyhedral set $E$ is contained in the union of finitely many ($N$ say) affine hyperplanes $H_i$ for $i=1,\dots,N$.
We call the set $\partial E_i := \partial E \cap H_i$ a face of $E$.
By construction $\partial E = \cup_{i=1}^N \partial E_i$.
The intersection of two faces is an edge $e_{ij} = \partial E_i \cap \partial E_j$.
It is unfortunate but unavoidable that we use the term edge in two different contexts; either as a edge in a graph, or as a edge of a polyhedral set.
It should be clear from the context what is meant.
The intersection of edges is called a corner.


To reduce bias let us redefine the normalization on $GTV_n$ so that
\begin{equation} \label{eq:Rate:GTVn}
GTV_n(\mu) = \frac{1}{\epsilon_n} \frac{1}{n(n-1)} \sum_{i,j} \eta_{\epsilon_n}(\xi_i - \xi_j) \left| \mu(\xi_i) - \mu(\xi_j) \right|.
\end{equation}
For $\Xi_{ij} = \frac{1}{\eps_n} \eta_{\epsilon_n}(\xi_i - \xi_j) \left| \mu(\xi_i) - \mu(\xi_j) \right|$ one can write
\[ GTV_n(\mu) = \frac{1}{n(n-1)} \sum_{i,j} \Xi_{ij}. \]
Since $\xi_i$ are iid then
\[ \mathbb{E} \Xi_{ij} = \left\{ \begin{array}{ll} \mathbb{E} \Xi_{12} & \text{if } i\neq j \\ 0 & \text{if } i=j. \end{array} \right. \]
Hence $\mathbb{E}GTV_n(\mu) = \mathbb{E} \Xi_{12}$.
This would not be true for the normalization we considered previously.

For simplicity we make the following assumptions.
Assume $X = (0,1)^d$ where $d\geq 1$ and that $\rho\equiv 1$ on $X$.
We use an isotropic interaction potential $\eta = \mathbb{I}_{B(0,1)}$.
These assumptions simplify the calculations which allows one to have a better understanding of the methodology without the notational burden if one used more general assumptions.
We expect that the results in this section can be generalised to a wider class of interaction potentials $\eta$, spaces $X\subset \mathbb{R}^d$ and probability densities $\rho$.
We start with the convergence of the expectation.

\begin{theorem}
\label{thm:Rate:ConvMean}
Let $X=(0,1)^d$ with $d\geq 2$, $\rho\equiv 1$ and $\epsilon_n$ be any sequence converging to zero.
The data is distributed $\xi_i\iid \rho$ and let $\Psi_n=\{\xi_i\}_{i=1}^n$.
Define $GTV_n:L^1(\Psi_n)\to [0,\infty]$ by~\eqref{eq:Rate:GTVn} where the weights are given by $W_{ij}=\eta_{\epsilon_n}(\xi_i-\xi_j)$ and $\eta_{\epsilon_n}(x)=\frac{1}{\epsilon_n^d}\mathbb{I}_{|x|\leq \epsilon_n}$.
Define $TV(\cdot;\rho,\eta):L^1(X)\to [0,\infty]$ by~\textnormal{(\ref{eq:Prelim:TV:TV}-\ref{eq:Prelim:TV:sigma})}.
Let $\mu=\mathbb{I}_E$ be a polyhedral function.
Then
\[ \left| \mathbb{E} GTV_n(\mu) - TV(\mu;\rho,\eta) \right| = O(\epsilon_n). \]
\end{theorem}

Note that we do not need a lower bound on the decay of $\epsilon_n$.
By taking the expectation we are immediately in the continuous setting and therefore lose all the graphical structure.
In particular
\[ \mathbb{E} GTV_n(\mu) = \frac{1}{\epsilon_n} \int_{(0,1)^d} \int_{(0,1)^d} \eta_{\epsilon_n}(x-y) \left| \mu(x) - \mu(y) \right| \; \mathrm{d} y \; \mathrm{d} x \]
has no discrete structure.

Our proof shows that along faces of $E$ and sufficiently far from edges in some sense the expected graph total variation is equal to the total variation.
To be more precise if $E=\{ x \in X \, : \, w\cdot x >0 \}$ for some $w\in \mathbb{R}^d$ then $E$ has no edges and therefore $GTV_n(\mu) = TV(\mu;\rho,\eta)$.
The discrepancy between $\mathbb{E} GTV_n(\mu)$ and $TV(\mu;\rho,\eta)$, which we show is of order $\eps_n$, is a consequence of having to approximate along edges of $E$.

\begin{proof}[Proof of Theorem~\ref{thm:Rate:ConvMean}]
Let $\partial E = \cup_{i=1}^N \partial E_i$.
We first calculate $TV(\mu;\rho,\eta)$,
\[ TV(\mu;\rho,\eta) = \int_{\partial\{\mu=1\}} \sigma(n(x)) \; \mathrm{d} \mathcal{H}^{d-1}(x) = \sum_{i=1}^N \left|\partial E_i\right|_{\mathcal{H}^{d-1}} \sigma(n_i) \]
where $n_i$ is the outward unit normal for side $\partial E_i$ and we use $|\cdot|_{\mathcal{H}^{d-1}}$ to denote the $\mathcal{H}^{d-1}$ measure.
Observe
\[ \sigma = \sigma(n_i) = \int_{B(0,1)} |x\cdot n_i| \; \mathrm{d} x = \int_{B(0,1)} |x_d| \; \mathrm{d} x. \]
So $TV(\mu;\rho,\eta) = \sigma |\partial E|_{\mathcal{H}^{d-1}}$.

Consider the face $\partial E_i$, we will approximate this with a smaller face $\partial E_i^{(n)} \subset \partial E_i$ that will approximate the graph total variation to within $O(\eps_n)$.
Consider the set,
\[ R_i(s) = \Big\{ x + \eps_n t n \, : \, t \in [-1,1], n \text{ is normal to } \partial E_i \text{ and } x \in \partial E_i \text{ with } \mathrm{dist}(\partial(\partial E_i),x) > s \Big\} \]
where $\partial(\partial E_i)$ is the $d-2$ dimensional boundary of the face $\partial E_i$.
There exists $s_n=O(\eps_n)$ such that for all $i$
\[ R_i(s_n) \cap \left( \bigcup_{j\neq i} \partial E_j \right) = \emptyset. \]
Define
\begin{align*}
B_n & = \Big\{ x \, : \, \mathrm{dist}(x,\partial E) < \eps_n \Big\} \\
S_i^{(n)} & = R_i(s_n+\eps_n) \quad \quad \quad \quad T_i^{(n)} = R_i(s_n) \\
U_i^{(n)} & = \Big\{ x+ \eps_n t n \, : \, t \in [-1,1], n \text{ is normal to } \partial E_i \text{ and } x \in \partial E_i \text{ with } \mathrm{dist}(\partial(\partial E_i),x) \leq s_n + \eps_n \Big\} \\
\partial E_i^{(n)} & = \Big\{ x \in \partial E_i \, : \, \mathrm{dist}(x,\partial(\partial E_i)) > s_n +\eps_n \Big\}.
\end{align*}
By construction $\bigcup_{x\in S_i^{(n)}\cap E} \left[ E^c \cap B(x,\eps_n) \right] \subseteq \left( T_i^{(n)}\cap E^c \right)$ and $\bigcup_{x\in S_i^{(n)}\cap E^c} \left[ E \cap B(x,\eps_n) \right] \subseteq \left( T_i^{(n)}\cap E \right)$.
Now,
\begin{align*}
\mathbb{E} GTV_n(\mu) & = \frac{1}{\eps_n} \int_{B_n} \int_{B_n} \eta_{\eps_n}(x-y) \left|\mu(x) - \mu(y) \right| \, \mathrm{d} y \, \mathrm{d} x \\
 & = \frac{1}{\eps_n^{d+1}} \sum_{i=1}^N \int_{S_i^{(n)}\cap E} \int_{T_i^{(n)}\cap E^c} \mathbb{I}_{|x-y|\leq \eps_n} \, \mathrm{d} y \, \mathrm{d} x  + \frac{1}{\eps_n^{d+1}} \int_{(B_n\cap E)\setminus \bigcup_{i=1}^N S_i^{(n)}} \int_{E^c} \mathbb{I}_{|x-y|\leq \eps_n} \, \mathrm{d} y \, \mathrm{d} x \\
 & \quad \quad + \frac{1}{\eps_n^{d+1}} \sum_{i=1}^N \int_{S_i^{(n)}\cap E^c} \int_{T_i^{(n)}\cap E} \mathbb{I}_{|x-y|\leq \eps_n} \, \mathrm{d} y \, \mathrm{d} x + \frac{1}{\eps_n^{d+1}} \int_{(B_n\cap E^c)\setminus \bigcup_{i=1}^N S_i^{(n)}} \int_E \mathbb{I}_{|x-y|\leq \eps_n} \, \mathrm{d} y \, \mathrm{d} x.
\end{align*}
We have
\begin{align*}
\frac{1}{\eps_n^{d+1}} \int_{(B_n\cap E)\setminus \bigcup_{i=1}^N S_i^{(n)}} \int_{E^c} \mathbb{I}_{|x-y|\leq \eps_n} \, \mathrm{d} y \, \mathrm{d} x & \leq \frac{1}{\eps_n} \sum_{i=1}^N \int_{U_i^{(n)}} \int_{x + \eps_n z \in E^c} \mathbb{I}_{|z|\leq 1} \, \mathrm{d} z \, \mathrm{d} x \\
 & \leq \frac{1}{\eps_n} \sum_{i=1}^N \mathrm{Vol}(U_i^{(n)}) \mathrm{Vol}(B(0,1)) \\
 & = O(\eps_n)
\end{align*}
since $\mathrm{Vol}(U_i^{(n)})=O(\eps_n (s_n+\eps_n)) = O(\eps_n^2)$.
Similarly
\[ \frac{1}{\eps_n^{d+1}} \int_{(B_n\cap E^c)\setminus \bigcup_{i=1}^N S_i^{(n)}} \int_E \mathbb{I}_{|x-y|\leq \eps_n} \, \mathrm{d} y \, \mathrm{d} x = O(\eps_n). \]

For the remaining terms in the above expansion of $\mathbb{E}GTV_n(\mu)$ we consider the $i^\text{th}$ face.
After a suitable change of coordinates we can assume that $\partial E_i \subset \{ x \, : \, x_1 = 0 \}$ and $n_i = (1,0,\dots,0)$.
Then
\begin{align*}
\frac{1}{\eps_n} \int_{S_i^{(n)}\cap E} \int_{T_i^{(n)}\cap E^c} \mathbb{I}_{|x-y|\leq \eps_n} \, \mathrm{d} y \, \mathrm{d} x & = \frac{1}{\eps_n} \int_{\partial E_i^{(n)}} \int_{-\eps_n}^0 \int_{y_1>0} \mathbb{I}_{|x-y|\leq \eps_n} \, \mathrm{d} y \, \mathrm{d} x_1 \, \mathrm{d} x_{2:d} \\
 & = \frac{1}{\eps_n} \int_{\partial E_i^{(n)}} \int_{-\eps_n}^0 \int_{x_1=\eps_n z_1 >0} \mathbb{I}_{|z|\leq 1} \, \mathrm{d} z \, \mathrm{d} x_1  \, \mathrm{d} x_{2:d} \\
 & = \frac{1}{\eps_n} \int_{\partial E_i^{(n)}} \int_{-\eps_n}^0 \int_{|z|\leq 1} \mathbb{I}_{x_1 +\eps_n z_1>0} \, \mathrm{d} z \, \mathrm{d} x_1 \, \mathrm{d} x_{2:d} \\
 & = \int_{\partial E_i^{(n)}} \int_{|z|\leq 1} z_1 \mathbb{I}_{z_1>0} \, \mathrm{d} z \, \mathrm{d} x_{2:d-1} \\
 & = \frac{1}{2} |\partial E_i^{(n)}|_{\mathcal{H}^{d-1}} \int_{|z|\leq 1} |z_1| \, \mathrm{d} z \\
 & = \frac{\sigma}{2} |\partial E_i^{(n)}|_{\mathcal{H}^{d-1}} \\
 & = \frac{\sigma}{2} |\partial E_i|_{\mathcal{H}^{d-1}} + O(\eps_n).
\end{align*}
Analogously,
\[ \frac{1}{\eps_n^{d+1}} \sum_{i=1}^N \int_{S_i^{(n)}\cap E^c} \int_{T_i^{(n)}\cap E} \mathbb{I}_{|x-y|\leq \eps_n} \, \mathrm{d} y \, \mathrm{d} x = \frac{\sigma}{2} |\partial E_i|_{\mathcal{H}^{d-1}} + O(\eps_n). \]
%
%
Collecting terms, we have shown
\[ \mathbb{E} GTV_n(\mu) = TV(\mu;\rho,\eta) + O(\epsilon_n) \]
which completes the proof.
\end{proof}

The above theorem established convergence in mean of $GTV_n(\mu^n)$ to $TV(\mu;\rho,\eta)$. 
A natural next step is to establish convergence in mean square.
The next theorem gives the asymptotic expansion of $\mathbb{E}\left| GTV_n(\mu) - TV(\mu;\rho,\eta) \right|^2$.
We order the expansion so that dominant terms come first (the ordering follows from the scaling given by Assumption~3 in Definition~\ref{def:MainRes:Admissible}).
The dominant term depends on the bias $\alpha_n = \mathbb{E} GTV_n(\mu) - TV(\mu;\rho,\eta)$ which is $O(\eps_n)$ due to the approximation along the edges of each face of $E$.
The complexity of refining this approximation, i.e. finding the constant $c$ such that $\alpha_n = c\eps_n + \text{h.o.t.}$ goes beyond the scope of the paper.

\begin{theorem}
\label{thm:Rate:ConvMeanSq}
Under the same conditions as Theorem~\ref{thm:Rate:ConvMean}
\begin{align*}
\mathbb{E} \left| GTV_n(\mu) - TV(\mu;\rho,\eta) \right|^2 & = - 2\alpha_n TV(\mu;\rho,\eta) + \frac{4(n-2)|\partial E|_{\mathcal{H}^{d-1}} V}{n(n-1) \eps_n} + \frac{2 TV(\mu;\rho,\eta)}{n(n-1)\eps_n^{d+1}} \\
 & \quad \quad + \frac{(n-2)(n-3) \alpha_n^2}{n(n-1)} + \frac{(6-4n)TV(\mu;\rho,\eta)^2}{n(n-1)} + O\left(\frac{1}{n}\right) \\
 & \quad \quad + O\left(\frac{1}{\eps_n^dn^2}\right) + \frac{2\alpha_n TV(\mu;\rho,\eta)}{n(n-1)}
\end{align*}
where $\alpha_n = \mathbb{E} GTV_n(\mu) - TV(\mu;\rho,\eta) = O(\epsilon_n)$ is the bias and
\[ V = \frac{1}{2} \int_{B(0,1)} \int_{B(0,1)} \min\{|z_d|,|y_d|\} \; \mathrm{d} z \; \mathrm{d} y. \]
\end{theorem}

\begin{proof}
We can write
\begin{align*}
\mathbb{E} \left| GTV_n(\mu) - TV(\mu;\rho,\eta) \right|^2 & = \mathbb{E} GTV_n(\mu)^2 + TV(\mu;\rho,\eta)^2 - 2TV(\mu;\rho,\eta) \mathbb{E} GTV_n(\mu) \\
 & = \mathbb{E} GTV_n(\mu)^2 - TV(\mu;\rho,\eta)^2 - 2TV(\mu;\rho,\eta) \alpha_n.
\end{align*}
Let $\Xi_{ij} = \frac{1}{\epsilon_n} \eta_{\epsilon_n}(\xi_i-\xi_j) \left| \mu(\xi_i)-\mu(\xi_j)\right|$ then
\[ GTV_n(\mu) = \frac{1}{n(n-1)} \sum_{i,j} \Xi_{ij} \quad \text{and} \quad GTV_n(\mu)^2 = \frac{1}{n^2(n-1)^2} \sum_{i,j,k,l} \Xi_{ij}\Xi_{kl}. \]
Let $i,j,k,l$ be distinct, then $GTV_n(\mu)^2$ has the following contributions:
\begin{enumerate}
\item[A.] $2n(n-1)$ terms consisting of $\Xi_{ij}^2$,
\item[B.] $4n(n-1)(n-2)$ terms consisting of $\Xi_{ij}\Xi_{ik}$ and
\item[C.] $n(n-1)(n-2)(n-3)$ terms consisting of $\Xi_{ij}\Xi_{kl}$.
\end{enumerate}

For~C we use independence of $\Xi_{ij}$ with $\Xi_{kl}$ to write
\[ \mathbb{E} \Xi_{ij}\Xi_{kl} = \mathbb{E}\Xi_{ij} \mathbb{E}\Xi_{kl} = \left(\mathbb{E} GTV_n(\mu) \right)^2 = \left(TV(\mu;\rho,\eta) + \alpha_n \right)^2 = TV(\mu;\rho,\eta)^2 + 2\alpha_n TV(\mu;\rho,\eta) + \alpha_n^2. \]
For~A:
\begin{align*}
\mathbb{E} \Xi_{ij}^2 & = \frac{1}{\epsilon_n^{2d+2}} \int_{(0,1)^d} \int_{(0,1)^d} \mathbb{I}_{|x-y|\leq \epsilon_n} \left| \mu(x) - \mu(y) \right| \; \mathrm{d} x \; \mathrm{d} y \\
 & = \frac{1}{\epsilon_n^{d+1}} \mathbb{E} GTV_n(\mu) \\
 & = \frac{1}{\epsilon_n^{d+1}} TV(\mu;\rho,\eta) + O\left(\frac{1}{\epsilon_n^d}\right) \quad \text{by Theorem~\ref{thm:Rate:ConvMean}.}
\end{align*}
Now we consider~B.
We have
\begin{align*}
\mathbb{E} \Xi_{ij} \Xi_{ik} & = \frac{1}{\eps_n^{2d+2}} \int_{(0,1)^d} \int_{(0,1)^d} \int_{(0,1)^d} \mathbb{I}_{|x-y|\leq \eps_n} \mathbb{I}_{|x-w|\leq \eps_n} |\mu(x) - \mu(y)| |\mu(x) - \mu(w)| \, \mathrm{d} w \, \mathrm{d} y \, \mathrm{d} x \\
 & = \frac{1}{\eps_n^{2d+2}} \sum_{i=1}^N \int_{S_i^{(n)}\cap E} \int_{T_i^{(n)}\cap E^c} \int_{T_i^{(n)}\cap E^c} \mathbb{I}_{|x-y|\leq \eps_n} \mathbb{I}_{|x-w|\leq \eps_n} \, \mathrm{d} w \, \mathrm{d} y \, \mathrm{d} x \\
 & \quad \quad + \frac{1}{\eps_n^{2d+2}} \int_{(B_n\cap E) \setminus \cup_{i=1}^N S_i^{(n)}} \int_{E^c} \int_{E^c} \mathbb{I}_{|x-y|\leq \eps_n} \mathbb{I}_{|x-w|\leq \eps_n} \, \mathrm{d} w \, \mathrm{d} y \, \mathrm{d} x \\
 & \quad \quad + \frac{1}{\eps_n^{2d+2}} \sum_{i=1}^N \int_{S_i^{(n)}\cap E^c} \int_{T_i^{(n)}\cap E} \int_{T_i^{(n)}\cap E} \mathbb{I}_{|x-y|\leq \eps_n} \mathbb{I}_{|x-w|\leq \eps_n} \, \mathrm{d} w \, \mathrm{d} y \, \mathrm{d} x \\
 & \quad \quad + \frac{1}{\eps_n^{2d+2}} \int_{(B_n\cap E^c) \setminus \cup_{i=1}^N S_i^{(n)}} \int_E \int_E \mathbb{I}_{|x-y|\leq \eps_n} \mathbb{I}_{|x-w|\leq \eps_n} \, \mathrm{d} w \, \mathrm{d} y \, \mathrm{d} x.
\end{align*}
Now,
\begin{align*}
& \frac{1}{\eps_n^{2d+2}} \int_{(B_n\cap E) \setminus \cup_{i=1}^N S_i^{(n)}} \int_{E^c} \int_{E^c} \mathbb{I}_{|x-y|\leq \eps_n} \mathbb{I}_{|x-w|\leq \eps_n} \, \mathrm{d} w \, \mathrm{d} y \, \mathrm{d} x \\
& \quad \quad \quad \quad = \frac{1}{\eps_n^2} \int_{(B_n\cap E)\setminus \cup_{i=1}^N S_i^{(n)}} \int_{x+\eps_n y \in E^c} \int_{x+\eps_n w \in E^c} \mathbb{I}_{|z|\leq 1} \mathbb{I}_{|w|\leq 1} \, \mathrm{d} w \, \mathrm{d} y \, \mathrm{d} x \\
& \quad \quad \quad \quad \leq \frac{1}{\eps_n^2} \left(\mathrm{Vol}(B(0,1)) \right)^2 \mathrm{Vol}((B_n\cap E)\setminus \cup_{i=1}^N S_i^{(n)}) \\
& \quad \quad \quad \quad = O(1).
\end{align*}
where $B_n, S_i^{(n)}$ and $T_i^{(n)}$ are as in the proof of Theorem~\ref{thm:Rate:ConvMean}.
Considering each face individually, after rotating,
\begin{align*}
& \frac{1}{\eps_n^{2d+2}} \int_{S_i^{(n)}\cap E} \int_{T_i^{(n)}\cap E^c} \int_{T_i^{(n)}\cap E^c} \mathbb{I}_{|x-y|\leq \eps_n} \mathbb{I}_{|x-w|\leq \eps_n} \, \mathrm{d} w \, \mathrm{d} y \, \mathrm{d} x \\
& \quad \quad \quad \quad = \frac{1}{\eps_n^2} \int_{S_i^{(n)}\cap E} \int_{x_d+\eps_n y_d>0} \int_{x_d+\eps_n w >0} \mathbb{I}_{|y|\leq 1} \mathbb{I}_{|w|\leq 1} \, \mathrm{d} w \, \mathrm{d} y \, \mathrm{d} x \\
& \quad \quad \quad \quad = \frac{1}{\eps_n^2} \int_{\partial E_i^{(n)}} \int_{-\eps_n}^0 \int_{|y|\leq 1} \int_{|w|\leq 1} \mathbb{I}_{x_d + \eps_n y_d > 0} \mathbb{I}_{x_d + \eps_n w_d >0} \, \mathrm{d} y \, \mathrm{d} x_d \, \mathrm{d} x_{1:d-1} \\
& \quad \quad \quad \quad = \frac{1}{\eps_n^2} \int_{\partial E_i^{(n)}} \int_{-\eps_n}^0 \int_{|w|\leq 1} \int_{|y|\leq 1} \mathbb{I}_{x_d> -\eps_n \min\{w_d,y_d\}} \, \mathrm{d} y \, \mathrm{d} w \, \mathrm{d} x_d \, \mathrm{d} x_{1:d-1} \\
& \quad \quad \quad \quad = \frac{1}{\eps_n} \int_{\partial E_i^{(n)}} \int_{|w|\leq 1} \int_{|y|\leq 1} \mathbb{I}_{\min\{ w_d,y_d\} >0} \min\{ w_d,y_d\} \, \mathrm{d} y \, \mathrm{d} w \, \mathrm{d} x_{1:d-1} \\
& \quad \quad \quad \quad = \frac{1}{4\eps_n} \int_{\partial E_i^{(n)}} \int_{|w|\leq 1} \int_{|y|\leq 1} \min\{ |w_d|, |y_d|\} \, \mathrm{d} y \, \mathrm{d} w \, \mathrm{d} x_{1:d-1} \\
& \quad \quad \quad \quad = \frac{|\partial E_i^{(n)}|_{\mathcal{H}^{d-1}} V}{2\eps_n} \\
& \quad \quad \quad \quad = \frac{|\partial E_i|_{\mathcal{H}^{d-1}} V}{2\eps_n} + O(1).
\end{align*}
Hence
\[ \mathbb{E} \Xi_{ij} \Xi_{ik} = \frac{|\partial E|_{\mathcal{H}^{d-1}} V}{\eps_n} + O(1). \]

Collecting terms implies the result of the theorem.
%
\end{proof}

If one defines
\begin{align*}
\kappa_1 & = 4|\partial E|_{\mathcal{H}^{d-1}} V \\
\kappa_2 & = 2TV(\mu;\rho,\eta)
\end{align*}
then we can conclude the asymptotic expansion given in Section~\ref{subsec:MainRes:Rate}.

\section*{Acknowledgements}

Part of this work was completed whilst MT was part of MASDOC at the University of Warwick and was supported by an EPSRC Industrial CASE Award PhD Studentship with Selex ES Ltd.
The authors would also like to thank Neil Cade (Selex ES Ltd.) and Adam Johansen (Warwick University) whose discussions enhanced this paper.

\appendix
\appendixpage

\section{Proof of Proposition~\ref{prop:MainRes:SatAss} \label{app:ProofPropMainResSatAss}}

\begin{proof}[Proof of Proposition~\ref{prop:MainRes:SatAss}]
For the first part assume the support of $\varphi$ is contained in $[0,M]$ and choose $N \leq M$ such that $\varphi(t) \geq \frac{\varphi(0)}{2}>0$ for all $0\leq t \leq N$ and let $0<\delta<N$.
First consider $|x|\geq N-\delta$.
If $|x-z|<\delta$ then
\[ |z| \leq |z-x| + |x| \leq \delta + |x| = |x| \left(1 + \frac{\delta}{|x|} \right) \leq |x| \left( 1+ \frac{\delta}{N-\delta} \right). \]
Set $c_\delta = 1$ and $\alpha_\delta = 1+\frac{\delta}{N-\delta}$, then as $\varphi$ is decreasing we have
\[ \eta(z) = \varphi(|z|) \geq \varphi(\alpha_\delta|x|) = c_\delta \eta(\alpha_\delta x). \]

Now if $|x|\leq N-\delta$ then $|z|\leq N \leq M$ and therefore
\[ \eta(z) = \phi(|z|) \geq \varphi(|x|) - L\left| |x| - |z| \right| \geq \varphi(|x|) - L|x-z| \geq \varphi(|x|) - L\delta = \varphi(|x|) \left( 1 - \frac{L\delta}{\varphi(|x|)} \right) \]
where $L$ is the Lipschitz constant of $\varphi$.
By definition of $N$ we have that $\varphi(|x|) \geq \frac{\varphi(0)}{2}$ and therefore
\[ \eta(z) \geq \varphi(|x|) \left( 1 - \frac{2L\delta}{\varphi(0)} \right). \]
Hence $\eta(z) \geq c_\delta \eta(x)$ where $c_\delta = 1-\frac{2L\delta}{\varphi(0)}$ and $\alpha_\delta=1$.

For the second case let $B(0,2m) \subset E \subset B(0,M)$ and $\delta^* = \mathrm{dist}_H(\partial E, B(0,m))$ where $\mathrm{dist}_H$ is the Hausdorff distance.
Clearly $\delta^* \geq m >0$.
Let
\[ \partial_{\delta^*} E = \left\{ x \in E \cup \partial E \, : \, \mathrm{dist}_H(x,\partial E) \leq \frac{\delta^*}{2} \right\}. \]
Note that if $x\in \partial_{\delta^*} E$ then $\mathrm{dist}_H(x,B(0,m)) \geq \frac{\delta^*}{2}$.
For any $x \in \partial_{\delta^*} E$ there exists a unique (by convexity) $\beta_x\geq 1$ such that $\beta_x x \in \partial E$.
Furthermore, $\beta_x = \frac{|\beta_x x|}{|x|} \leq \frac{M}{m}$.

Let $\delta \leq \delta^*$ and pick $z \in \partial_\delta^+ := \left\{ x \in E^c \, : \, \mathrm{dist}_H(z,\partial E) \leq \delta \right\}$ then for any $x\in B(z,\delta) \cup E$ we have
\begin{align*}
\beta_x - 1 & = \frac{1}{|x|} \left( \beta_x|x| - |z| + |z| - |x| \right) \\
 & \leq \frac{1}{m} \left( |\beta_x - x | + |z-x| \right) \\
 & \leq \frac{1}{m} \left( |\beta_x - x | + \delta \right).
\end{align*}
Now we construct the triangle given in Figure~\ref{fig:MainRes:Triangle}.
Applying the cosine formula one has,
\begin{align*}
\cos(\theta) & = \frac{|z|^2 + |x|^2 - |z-x|^2}{2|x||z|} \\
 & = 1 + \frac{|z|^2 + |x|^2 - |z-x|^2 - 2|x||z|}{2|x||z|} \\
 & \geq 1 - \frac{|z-x|^2}{2|x||z|} \\
 & \geq 1 - \frac{\delta^2}{2m^2}.
\end{align*}
And
\begin{align*}
|\beta_x x - z|^2 & = |z|^2 + |\beta_x x|^2 - 2|z| |\beta_x x| \cos(\theta) \\
 & \leq |z|^2 + |\beta_x x|^2 - 2|z| |\beta_x x| + \frac{|z| |\beta_x x| \delta^2}{m^2} \\
 & \leq \frac{|z| |\beta_x x|\delta^2}{m^2} \\
 & \leq \frac{M^3\delta^2}{m^3}.
\end{align*}
Therefore
\[ \beta_x - 1 \leq \frac{\delta}{m} \left( \left(\frac{M}{m}\right)^{\frac{3}{2}} + 1 \right). \]
Take $\alpha_\delta = \sup_{z\in \partial_{\delta}^+ E} \sup_{x\in B(z,\delta)\cap E} \beta_x$.
Then $1\leq \alpha_\delta \leq 1+ \frac{\delta}{m} \left( \left(\frac{M}{m}\right)^{\frac{3}{2}} + 1 \right) \to 1$ and by construction for any $x\in B(z,\delta)\cap E$ we have $\alpha x \not\in E$.
This implies $\eta(\alpha_\delta x) = 0$ and therefore $\eta(z) \geq \eta(\beta_x x)$.

The other cases, $x \in B(z,\delta) \cap E^c$ and when $z \in E$ or $\mathrm{dist}_H(z,\partial E) > \delta$ are trivial.
\begin{figure}
\centering
\scriptsize
\setlength\figureheight{5cm}
\setlength\figurewidth{5cm}
\begin{tikzpicture}
\begin{axis}[%
width=\figurewidth,
height=\figureheight,
scale only axis,
xmin=0,
xmax=8,
ymin=0,
ymax=8,
hide axis,
]
\draw (axis cs: 1,4) node[left] {A} -- (axis cs: 4,5) node[midway,below,xshift=3] {$|z-x|$} node[right] {C};
\draw (axis cs: 4,1) node[below] {O} -- (axis cs: 1,4) node[midway,below,xshift=-3] {$|z|$};
\draw (axis cs: 4,1) -- (axis cs: 4,5) node[midway,right] {$|x|$};
\draw (axis cs: 1,4) -- (axis cs: 4,7) node[midway,above,xshift=-15] {$|\beta_x x-z|$} node[above] {B};
\draw (axis cs: 4,5) -- (axis cs: 4,7) node[midway,right] {$(\beta_x-1)|x|$};
\addplot [
color=black,
solid,
forget plot
]
table[row sep=crcr]{
4.0000 1.5000 \\
3.9608 1.4985 \\
3.9218 1.4938 \\
3.8833 1.4862 \\
3.8455 1.4755 \\
3.8087 1.4619 \\
3.7730 1.4455 \\
3.7388 1.4263 \\
3.7061 1.4045 \\
3.6753 1.3802 \\
3.6464 1.3536 \\
};
\draw (axis cs: 3.8,1.7) node {$\theta$};
\end{axis}
\end{tikzpicture}
\caption{
Bound for $|\beta_x x - z|$ in the proof of Proposition~\ref{prop:MainRes:SatAss}.
}
\label{fig:MainRes:Triangle}
\end{figure}
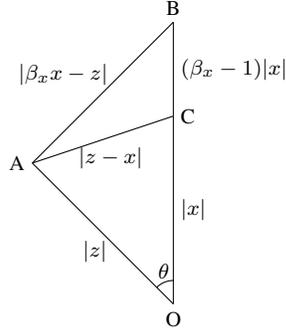
\end{proof}

\section{Proof of Corollary~\ref{cor:MainRes:Isotropic} \label{app:ProofCorMainResIsotropic}}

\begin{proof}[Proof of Corollary~\ref{cor:MainRes:Isotropic}]
The compactness property holds analogously to Proposition~\ref{prop:Compact:ComSecondTermData}.
For the $\Gamma$-convergence we let $\eta=\varphi\circ\pi=\varphi(|x|)$.
Since $\varphi$ is in $L^1([0,\infty))$ (it is bounded, measurable and with compact support) then we can approximate $\varphi$ by $\varphi_k$ a monotonically increasing sequence (in $k$) of functions such that $0\leq \varphi_k\leq \varphi$, $\varphi_k\to \varphi$ pointwise, $\varphi_k$ is Lipschitz, decreasing and $\varphi_k>0$.
By Proposition~\ref{prop:MainRes:SatAss} and Theorem~\ref{thm:MainRes:Cons} for any $\mu\in L^1(X)$ there exists a sequence $\mu_n$ such that $\mu_n\to \mu$ in $TL^1$ and $\limsup_{n\to \infty} \mathcal{E}_n(\mu_n;\varphi_k\circ\pi) \leq \mathcal{E}_\infty(\mu;\varphi_k\circ\pi)$.
Therefore,
\begin{align*}
\limsup_{n\to \infty} \mathcal{E}_n(\mu_n;\varphi\circ\pi) & \geq \limsup_{n\to \infty} \mathcal{E}_n(\mu_n;\varphi_k\circ\pi) \\
 & \geq \mathcal{E}_\infty(\mu;\varphi_k\circ\pi) \\
 & \to \mathcal{E}_\infty(\mu;\varphi\circ\pi) \quad \text{as } k\to \infty \text{ by the monotone convergence theorem.}
\end{align*}

Similarly for the liminf inequality we take a sequence $\varphi_k$ monotonically decreasing sequence of functions such that $\varphi_k\geq \varphi$, $\varphi_k\to \varphi$ pointwise, $\varphi_k$ is Lipschitz, decreasing and with compact support.
Then for any $\mu_n\to \mu$ in $TL^1$ we have
\begin{align*}
\liminf_{n\to \infty} \mathcal{E}_n(\mu_n;\varphi\circ\pi) & \leq \liminf_{n\to \infty} \mathcal{E}_n(\mu_n;\varphi_k\circ\pi) \\
 & = \mathcal{E}_\infty(\mu;\varphi_k\circ\pi) \\
 & \to \mathcal{E}_\infty(\mu;\varphi\circ\pi) \quad \text{as } k\to \infty \text{ by the monotone convergence theorem.}
\end{align*}
This completes the proof.
\end{proof}

\nocite{baldi01,chambolle11,bresson13,jylha15,ambrosio03,strogatz01,trillos14a,vangennip14,garcia15}
\bibliographystyle{plain}
\bibliography{references}

\end{document}